\documentclass[10pt]{amsart}
\usepackage{amssymb}

\newtheorem{thm}{Theorem}[section]
\newtheorem{lemma}[thm]{Lemma}
\newtheorem{proposition}[thm]{Proposition}

\newtheorem{definition}[thm]{Definition}
\newtheorem{corollary}[thm]{Corollary}
\newtheorem{claim}[thm]{Claim}

\newtheorem{notation}[thm]{Notation}

\newcommand{\p}{\mathbb{P}}
\newcommand{\q}{\mathbb{Q}}

\newcommand{\cf}{\mathrm{cf}}
\newcommand{\cof}{\mathrm{cof}}
\newcommand{\dom}{\mathrm{dom}}

\begin{document}

\title{Indestructible Guessing Models and the Continuum}

\author{Sean Cox and John Krueger}

\address{Sean Cox \\ Department of Mathematics and Applied Mathematics \\
Virginia Commonwealth University \\ 
1015 Floyd Avenue \\ 
PO Box 842014 \\ 
Richmond, Virginia 23284}
\email{scox9@vcu.edu}

\address{John Krueger \\ Department of Mathematics \\ 
University of North Texas \\
1155 Union Circle \#311430 \\
Denton, TX 76203}
\email{jkrueger@unt.edu}

\date{October 2015; revised December 2016}



\thanks{Part of this work was completed in June 2015 under a collaboration 
sponsored by the first author's Simons Foundation grant 318467.}

\thanks{The research of the 
second author was partially supported by NSF grant DMS-1464859.}

\begin{abstract}
We introduce a stronger version of an $\omega_1$-guessing model, 
which we call an indestructibly $\omega_1$-guessing model. 
The principle $\textsf{IGMP}$ states that there are stationarily many 
indestructibly $\omega_1$-guessing models. 
This principle, which follows from $\textsf{PFA}$, captures many of the 
consequences of $\textsf{PFA}$, including the Suslin hypothesis and the 
singular cardinal hypothesis. 
We prove that $\textsf{IGMP}$ is consistent with the continuum 
being arbitrarily large.
\end{abstract}

\maketitle

The idea of an $\omega_1$-guessing model was introduced 
by Viale-Weiss \cite{vialeweiss}, who showed that the combinatorial 
principle 
$\textsf{ISP}(\omega_2)$ of Weiss \cite{weiss} 
is equivalent to the existence of stationarily 
many $\omega_1$-guessing models in $P_{\omega_2}(H(\theta))$, 
for all cardinals $\theta \ge \omega_2$. 
They showed that $\textsf{PFA}$ implies $\textsf{ISP}(\omega_2)$, and 
in turn $\textsf{ISP}(\omega_2)$ implies many of the consequences 
of $\textsf{PFA}$, including the failure of square principles. 
In \cite{vialeweiss} it was asked whether $\textsf{ISP}(\omega_2)$ 
determines the value of the continuum. 
We answered this question negatively in \cite{jk26}, by showing that 
$\textsf{ISP}(\omega_2)$ is consistent with $2^\omega$ having any 
value of uncountable cofinality greater than $\omega_1$.

In this paper we introduce a stronger kind of guessing model, which we 
call an \emph{indestructibly $\omega_1$-guessing model}. 
We also introduce a new principle, denoted by $\textsf{IGMP}$, which asserts the 
existence of stationarily many indestructibly $\omega_1$-guessing models 
in $P_{\omega_2}(H(\theta))$, for any cardinal 
$\theta \ge \omega_2$. 
As before, 
$\textsf{PFA}$ implies $\textsf{IGMP}$, 
but $\textsf{IGMP}$ captures more of the consequences of $\textsf{PFA}$ 
than does $\textsf{ISP}(\omega_2)$, 
including the Suslin hypothesis. 
As with the principle $\textsf{ISP}(\omega_2)$, a natural question is 
whether the stronger principle $\textsf{IGMP}$ determines the value of 
the continuum. 
The main result of this paper is that $\textsf{IGMP}$ is consistent with 
$2^\omega$ being equal to any $\lambda \ge \omega_2$ with 
cofinality at least $\omega_2$.

\bigskip

In Section 1, we review material which will be necessary for 
reading the paper. 
In Section 2, we discuss $\omega_1$-guessing models, and give new 
proofs of several previously known consequences of $\textsf{ISP}(\omega_2)$. 
In Section 3, we introduce indestructibly $\omega_1$-guessing models, 
the principle $\textsf{IGMP}$, and derive some consequences of 
$\textsf{IGMP}$, including Suslin's hypothesis and $\textsf{SCH}$.

In Section 5, we review the ideas of strong genericity and the strongly proper 
collapse. We also prove a new theorem about the preservation of 
strong properness after proper forcing. 
In Sections 6 and 7, we carefully develop a finite support iteration of 
specializing forcings, and prove that a certain quotient of such an 
iteration has the $\omega_1$-approximation property. 
In Section 8, we prove our consistency result, that $\textsf{IGMP}$ 
is consistent with the continuum being greater than $\omega_2$.

\section{Preliminaries}

We review some background material which will 
be necessary for understanding the paper. 
We assume that the reader is already familiar with forcing, iterated forcing, 
and proper forcing. 

If $\kappa$ is a regular cardinal and $X$ is a set, 
$P_{\kappa}(X)$ denotes the set $\{ a \subseteq X : |a| < \kappa \}$. 
The reader should be familiar with the basic definitions and facts regarding 
club and stationary subsets of $P_{\kappa}(X)$.

A \emph{tree} is a strict partial ordering $(T,<_T)$ such that for any $t \in T$, 
the set $\{ u \in T : u <_T t \}$ is well-ordered by $<_T$. 
We write $u \le_T v$ to mean that either $u <_T v$ or $u = v$. 
We sometimes say that $T$ is a tree without explicitly mentioning its partial order, 
which is always denoted by $<_T$. 
For an ordinal $\alpha$, $T_\alpha$ is the set of $t \in T$ such that 
the set $\{ u \in T : u <_T t \}$, ordered by $<_T$, has order type $\alpha$. 
The set $T_\alpha$ is called \emph{level $\alpha$ of $T$}. 
The \emph{height of $T$} is the least $\delta$ such that $T_\delta$ is empty. 
Let $T \restriction \beta := \bigcup \{ T_\alpha : \alpha < \beta \}$. 
A set $b$ is a \emph{branch of $T$} if it is a maximal linearly 
ordered subset of $T$. 

For a regular uncountable cardinal $\kappa$, a 
\emph{$\kappa$-Aronszajn tree} is a tree of height $\kappa$, all of whose levels 
have size less than $\kappa$, which has no branch of length $\kappa$. 
A \emph{weak $\kappa$-Kurepa tree} is a tree of height and size $\kappa$ 
which has more than $\kappa$ many branches of length $\kappa$.

Let $T$ be a tree. 
A \emph{specializing function for $T$} is a function 
$f : T \to \omega$ such that for all $x, y \in T$, 
if $x <_T y$ then $f(x) \ne f(y)$. 
Note that if $T$ has a specializing function, then $T$ has no branches 
of length $\omega_1$. 
On the other hand, suppose that $T$ is a tree 
which has no branches of length $\omega_1$.  
Define $P(T)$ as the forcing poset, ordered by reverse inclusion, 
whose conditions are 
finite functions $p$ from a subset of $T$ into $\omega$, such that 
for all $x, y \in \dom(p)$, if $x <_T y$ in $\dom(p)$, then $p(x) \ne p(y)$. 
Then $P(T)$ is $\omega_1$-c.c., 
and if $G$ is a $V$-generic filter for $P(T)$, then 
$\bigcup G$ is a specializing function for $T$ (\cite{baumgartner2}).

We will frequently use the product lemma. 
This result says that if $\p$ and $\q$ are forcing posets, then the $V$-generic 
filters for $\p \times \q$ are exactly those filters of the form 
$G \times H$, where 
$G$ is a $V$-generic filter for $\p$, and $H$ is a $V[G]$-generic 
filter for $\q$. 
Moreover, in that case $H$ is a $V$-generic filter for $\q$, 
$G$ is a $V[H]$-generic filter for $\p$, and 
$V[G \times H] = V[G][H] = V[H][G]$.

Let $\p$ and $\q$ be forcing posets, where $\p$ is a suborder of $\q$. 
We say that $\p$ is a \emph{regular suborder} of $\q$ if (a) for all $p$ and $q$ in $\p$, 
if $p$ and $q$ are incompatible in $\p$, then $p$ and $q$ are incompatible in $\q$, 
and (b) if $A$ is a maximal antichain of $\p$, then $A$ is predense in $\q$.

Let $\p$ be a regular suborder of $\q$, and let $G$ be a $V$-generic 
filter on $\p$. 
In $V[G]$, let $\q / G$ be the forcing poset consisting of conditions 
$q \in \q$ such that for all $s \in G$, 
$q$ and $s$ are compatible in $\q$, with the same ordering as $\q$. 
Then $\p * (\q / \dot G_{\p})$ is forcing equivalent to $\q$. 
Moreover:

\begin{lemma}
Let $\p$ be a regular suborder of $\q$.
Suppose that $G$ is a $V$-generic filter on $\p$, and $H$ is a 
$V[G]$-generic filter on $\q / G$. 
Then $H$ is a $V$-generic filter on $\q$, $H \cap \p = G$, 
and $V[G][H] = V[H]$.

Conversely, if $H$ is a $V$-generic filter on $\q$, then 
$H \cap \p$ is a $V$-generic filter on $\p$, 
$H$ is a $V[H \cap \p]$-generic filter on $\q / (H \cap \p)$, 
and $V[H] = V[H \cap \p][H]$.
\end{lemma}

\begin{proof}
See \cite[Lemma 1.6]{jk26}
\end{proof}

\begin{lemma}
Let $\p$ be a regular suborder of $\q$. 
Then for all $q \in \q$, there is $s \in \p$ such that for all $t \le s$ 
in $\p$, $q$ and $t$ are compatible in $\q$. 
Moreover, this property of $s$ is equivalent to $s$ forcing in $\p$ 
that $q$ is in $\q / \dot G_\p$.
\end{lemma}

\begin{proof}
See \cite[Lemmas 1.1, 1.3]{jk26}.
\end{proof}

\begin{lemma}
Let $\p$ and $\q$ be forcing posets, and assume that $\p$ is a 
regular suborder of $\q$. 
If $D$ is a dense subset of $\q$, then $\p$ forces that 
$D \cap (\q / \dot G_\p)$ 
is a dense subset of $\q / \dot G_\p$.
\end{lemma}

\begin{proof}
See \cite[Lemma 1.5]{jk26}.
\end{proof}

\begin{lemma}
Let $\p$ and $\q$ be forcing posets, and assume that $\p$ is a 
regular suborder of $\q$. 
Let $G$ be a $V$-generic filter on $\p$. 
Suppose that $s \in G$ and $p \in \q / G$. 
Then $s$ and $p$ are compatible in $\q / G$.
\end{lemma}

\begin{proof}
If not, then there is $t \in G$ such that 
$t$ forces in $\p$ that (a) $p$ is in $\q / \dot G_\p$, and 
(b) $p$ and $s$ are incompatible in $\q / \dot G_\p$. 
Fix $u \in G$ with $u \le s, t$.

As $p \in \q / G$ and $u \in G$, fix $v \le u, p$ in $\q$. 
By Lemma 1.2, fix $w \in \p$ such that for all $z \le w$ in $\p$, 
$v$ and $z$ are compatible in $\q$. 
Then in particular, $v$ and $w$ are compatible in $\q$, and since 
$v \le u$, $u$ and $w$ are compatible in $\q$. 
As $\p$ is a regular suborder of $\q$ and $u$ and $w$ are in $\p$, 
it follows that $u$ and $w$ are 
compatible in $\p$. 
Fix $y \le w, u$ in $\p$.

Since $y \le w$, every extension of $y$ in $\p$ is compatible with $v$ 
in $\q$. 
By Lemma 1.2, $y$ forces in $\p$ that $v \in \q / \dot G_\p$. 
Now $v \le u, p$ in $\q$. 
Since $u \le s$, we have that $v \le s, p$ in $\q$. 
But $y$ forces that $v \in \q / \dot G_\p$, so $y$ forces that 
$s$ and $p$ are compatible in $\q / \dot G_\p$. 
This contradicts the choice of $t$ 
and the fact that $y \le t$.
\end{proof}

Let $\p$ and $\q$ be forcing posets. 
A function $f : \p \to \q$ is a \emph{regular embedding} 
if (a) for all $p, q \in \p$, if $q \le p$ in $\p$, then 
$f(q) \le f(p)$ in $\q$; (b) for all $p, q \in \p$, 
if $p$ and $q$ are incompatible in $\p$, then $f(p)$ and $f(q)$ are 
incompatible in $\q$; 
(c) if $A$ is a maximal antichain of $\p$, then $f[A]$ is predense in $\q$. 
Note that if $f : \p \to \q$ is a regular embedding, then $f[\p]$ is a regular 
suborder of $\q$. 
A function $f : \p \to \q$ is a \emph{dense embedding} if it satisfies 
(a) and (b) above, and $f[\p]$ is dense in $\q$. 
Note that any dense embedding is a regular embedding.

Assume that $j : V \to M$ is an elementary embedding 
with critical point $\kappa$, living in some outer model $W$ of $V$. 
Let $\p$ be a forcing poset in $V$ which is $\kappa$-c.c. 
We claim that $j \restriction \p$ is a regular embedding of $\p$ 
into $j(\p)$. 
Namely, the preservation of the order and incompatibility of conditions from 
$\p$ to $j(\p)$ follows from the elementarity of $j$. 
And if $A$ is a maximal antichain of $\p$, then by elementarity 
$M$ models that $j(A)$ is a maximal antichain of $j(\p)$. 
By upwards absoluteness, $j(A)$ is a maximal antichain of $j(\p)$. 
But since $\p$ is $\kappa$-c.c., $|A| < \kappa$, and therefore 
$j(A) = j[A]$. 
So $j[A]$ is a maximal antichain of $j(\p)$.

A function $\pi : \q \to \p$, where $\p$ and $\q$ are forcing posets, 
is called a \emph{projection mapping} if (a) $\pi(1_\q) = 
1_\p$, (b) $q \le p$ in $\q$ implies that $\pi(q) \le \pi(p)$ 
in $\p$, and (c) whenever $p \le \pi(q)$ in $\p$, then there is some 
$r \le q$ in $\q$ such that $\pi(r) \le p$ in $\p$.
 
If $\pi : \q \to \p$ is a projection mapping and $G$ is a $V$-generic filter on $\p$, 
then in $V[G]$ we can define the forcing poset 
$\q / G$ whose conditions are those $q \in \q$ such 
that $\pi(q) \in G$, with the same ordering as $\q$. 
Then $\q$ is forcing equivalent to $\p * (\q / \dot G_\p)$.

\begin{lemma}
Let $\p$ be a suborder of $\q$. 
Suppose that there exists a projection mapping $\pi : \q \to \p$ satisfying that 
(i) $\pi(p) = p$ for all $p \in \p$, and 
(ii) $q \le \pi(q)$ for all $q \in \q$. 
Then $\p$ is a regular suborder of $\q$. 
Moreover, if $G$ is a $V$-generic filter on $\p$, then the two kinds of quotients 
$\q / G$ in $V[G]$ are the same.
\end{lemma}

\begin{proof}
Suppose that $p$ and $q$ are in $\p$, and $p$ and $q$ are compatible 
in $\q$. 
Let $s \le p, q$ in $\q$. 
Then $\pi(s) \le \pi(p) = p$ and $\pi(s) \le \pi(q) = q$. 
So $p$ and $q$ are compatible in $\p$.

Let $A$ be a maximal antichain of $\p$, and we will 
show that $A$ is predense in $\q$. 
Let $q \in \q$, and we will show that $q$ is compatible with 
some condition in $A$. 
Then $\pi(q) \in \p$, so as $A$ is a maximal antichain of $\p$, we can 
fix $s \in A$ and $r \in \p$ such that $r \le \pi(q), s$. 
Since $\pi$ is a projection mapping, we can fix $t \le q$ in $\q$ such that 
$\pi(t) \le r$. 
But then $t \le q$, and $t \le \pi(t) \le r \le s$, so $q$ is compatible 
with $s$.

Let $G$ be a $V$-generic filter on $\p$. 
We will prove that for all $q \in \q$, $q$ is compatible in $\q$ with 
every condition in $G$ iff $\pi(q) \in G$. 
Assume that $q$ is compatible with every condition in $G$. 
Since $G$ is a $V$-generic filter on $\p$, to show that $\pi(q) \in G$ 
it suffices to show 
that $\pi(q)$ is compatible in $\p$ with every condition in $G$. 
So let $s \in G$, and we will show that $\pi(q)$ and $s$ 
are compatible in $\p$. 
By assumption, $s$ and $q$ are compatible in $\q$, so fix 
$t \le q, s$. 
Then $\pi(t) \le \pi(q)$ and $\pi(t) \le \pi(s) = s$. 
Hence $\pi(q)$ and $s$ are compatible in $\p$.

Conversely, assume that $\pi(q) \in G$, and we will show that $q$ is compatible 
in $\q$ with every condition in $G$. 
Fix $s \in G$. 
Then there is $t \in G$ with $t \le \pi(q), s$. 
Since $\pi$ is a projection mapping, there is $u \le q$ such that 
$\pi(u) \le t$. 
Then $u \le q$ and $u \le \pi(u) \le t \le s$, so $q$ and $s$ 
are compatible in $\q$.
\end{proof}

A pair $(V,W)$ of transitive sets or classes with $V \subseteq W$ 
is said to have 
the \emph{$\omega_1$-covering property} if for every set $a \in W$ 
which is a subset of $V \cap On$ and which $W$ models is countable, there is 
a set of ordinals 
$b \in V$ which $V$ models is countable such that $a \subseteq b$. 
A forcing poset $\p$ has the \emph{$\omega_1$-covering property} if $\p$ 
forces that $(V,V[\dot G_\p])$ has the $\omega_1$-covering property.

For a set or class $N$, a set $d \subseteq N$ 
is said to be \emph{countably approximated by $N$} 
if for any set $a$ in $N$ which $N$ models is countable, 
$d \cap a \in N$. 
A pair $(V,W)$ of transitive sets or classes with $V \subseteq W$ 
is said to have the 
\emph{$\omega_1$-approximation property} 
if whenever $d \in W$ is a bounded subset of $V \cap On$ which is countably 
approximated by $V$, we have that $d \in V$. 
A forcing poset $\p$ has the \emph{$\omega_1$-approximation property} 
if $\p$ forces that $(V,V[\dot G_\p])$ has the $\omega_1$-approximation property.

Note that if $\p$ has the $\omega_1$-approximation property, then 
$\p$ preserves $\omega_1$.

\begin{lemma}
Suppose that $\p$ has the $\omega_1$-approximation property. 
Let $T$ be a tree. 
Suppose that $G$ is a $V$-generic filter on $\p$. 
Then any branch of $T$ in $V[G]$ whose length has 
uncountable cofinality is in $V$.
\end{lemma}

\begin{proof}
Without loss of generality, assume that the underlying set of $T$ 
is an ordinal. 
Let $b$ be a branch of $T$ in $V[G]$ whose length has 
uncountable cofinality. 
We claim that $b$ is countably approximated by $V$. 
Then $b \in V$ and we are done. 
Let $a$ be a countable set in $V$. 
Since the length of $b$ has uncountable cofinality, there is $y \in b$ 
such that $a \cap b \subseteq \{ x \in T : x <_T y \}$. 
Then $a \cap b  = a \cap \{ x \in T : x <_T y \}$, which is in $V$.
\end{proof}

\begin{lemma}
Suppose that $(V_0,V_1)$ and $(V_1,V_2)$ are pairs of transitive sets or 
classes which model $\textsf{ZFC}$, have the same ordinals, and 
satisfy that $V_0 \subseteq V_1 \subseteq V_2$. 
Assume that both pairs have the $\omega_1$-covering 
property and the $\omega_1$-approximation property. 
Then $(V_0,V_2)$ has the $\omega_1$-covering property 
and the $\omega_1$-approximation property.
\end{lemma}

\begin{proof}
It is easy to check that $(V_0,V_2)$ has the $\omega_1$-covering property. 
Let $d$ be a bounded subset of $V_0 \cap On$ in $V_2$ 
which is countably approximated by $V_0$. 
We will show that $d \in V_0$.

We claim that $d$ is countably approximated by $V_1$. 
Let $a$ be a countable set in $V_1$. 
Since $(V_0,V_1)$ has the $\omega_1$-covering property, 
we can fix a countable set $b$ in $V_0$ 
such that $a \subseteq b$. 
Since $d$ is countably approximated by $V_0$, 
$b \cap d$ is in $V_0$ and hence in $V_1$. 
Therefore $a \cap d = a \cap (b \cap d)$ is in $V_1$. 

Since $(V_1,V_2)$ has the $\omega_1$-approximation property, 
$d \in V_1$. 
As $d$ is countably approximated by $V_0$ and 
$(V_0,V_1)$ has the $\omega_1$-approximation property, $d \in V_0$.
\end{proof}

It follow that if $\p$ has the $\omega_1$-covering property 
and the $\omega_1$-approximation property, 
and $\p$ forces that $\dot \q$ has the 
$\omega_1$-covering property and the $\omega_1$-approximation property, 
then the 
two step iteration $\p * \dot \q$ has the $\omega_1$-covering 
property and the $\omega_1$-approximation property.

\begin{lemma}
Suppose that $V \subseteq W_0 \subseteq W$ are transitive sets or classes, 
and $(V,W)$ has the $\omega_1$-approximation property. 
Then $(V,W_0)$ has the $\omega_1$-approximation property.  
It follows that if $\p$ is a regular suborder of $\q$, and $\q$ has the 
$\omega_1$-approximation property, then $\p$ has the 
$\omega_1$-approximation property.
\end{lemma}

\begin{proof}
Straightforward.
\end{proof}

A set $N$ of size $\omega_1$ is said to be 
\emph{internally unbounded} 
if for any countable set $a \subseteq N$, there is a countable set $b \in N$ 
such that $a \subseteq b$. 
If $N \prec H(\theta)$ for some cardinal $\theta \ge \omega_2$, then easily $N$ 
is internally unbounded iff there exists a $\subseteq$-increasing sequence 
$\langle N_i : i < \omega_1 \rangle$ of countable sets in $N$ with union 
equal to $N$.

Finally, we will need to know some facts about the $Y$-c.c.\ property of 
forcing posets, 
which is a property introduced recently in \cite{ycc}. 
The actual definition of being $Y$-c.c.\ is beyond the scope of this paper. 
In \cite{ycc} it is proven that any $Y$-c.c.\ forcing poset is $\omega_1$-c.c.\ 
and has the $\omega_1$-approximation property, and 
any finite support iteration of $Y$-c.c.\ forcing posets is 
itself $Y$-c.c. 
Also, the forcing poset $P(T)$ defined earlier in this section, 
for adding a specializing function 
for a tree which has no branches of 
length $\omega_1$, is $Y$-c.c.

\section{Guessing Models and $\mathsf{GMP}$}

Guessing models were introduced by Viale-Weiss \cite{vialeweiss}.

\begin{definition}
Let $N$ be a set. 
A set $d \subseteq N$ is said to be \emph{$N$-guessed} if there exists 
$e \in N$ such that $d = e \cap N$. 
We say that $N$ is \emph{$\omega_1$-guessing} if for any set 
$d \subseteq N \cap On$ with $\sup(d) < \sup(N \cap On)$, 
if $d$ is countably approximated by $N$, then $d$ is $N$-guessed.
\end{definition}

A typical situation which we will consider is that $N$ is an elementary substructure 
of $H(\chi)$, for some cardinal $\chi \ge \omega_2$, 
and $|N| = \omega_1$.  
In the next section, we will also consider the case that $N$ is an elementary 
substructure in an inner model over which the universe is a generic 
extension.

There is a useful characterization of being $\omega_1$-guessing in terms 
of the approximation property.

\begin{lemma}
Let $N$ be an elementary substructure of $H(\chi)$, for some 
uncountable cardinal $\chi$. 
Then the following are equivalent:
\begin{enumerate}
\item $N$ is $\omega_1$-guessing;
\item the pair $(\overline{N},V)$ has the $\omega_1$-approximation property, 
where $\overline{N}$ is the transitive collapse of $N$.
\end{enumerate}
\end{lemma}

\begin{proof}
See \cite[Lemma 1.10]{jk26}.
\end{proof}

Next we prove two technical lemmas about $\omega_1$-guessing models.

\begin{lemma}
Let $N$ be in $P_{\omega_2}(H(\chi))$, where 
$\chi \ge \omega_2$ is a cardinal, and 
assume that $N \prec H(\chi)$ and $N$ is $\omega_1$-guessing. 
Then for any cardinal $\theta \in N$ with uncountable cofinality, 
$\cf(\sup(N \cap \theta)) = \omega_1$.
In particular, $\omega_1 \subseteq N$, and hence 
$N \cap \omega_2 \in \omega_2$.
\end{lemma}

\begin{proof}
Since $N$ has size at most $\omega_1$, 
$\cf(\sup(N \cap \theta)) \le \omega_1$. 
Suppose for a contradiction that $\cf(\sup(N \cap \theta)) = \omega$. 
Fix a sequence $\langle \alpha_n : n < \omega \rangle$ of ordinals 
in $N \cap \theta$ which is increasing and cofinal 
in $\sup(N \cap \theta)$. 

We claim that the set $\{ \alpha_n : n < \omega \}$ is countably 
approximated by $N$. 
Let $a \in N$ be countable, and we will show that 
$a \cap \{ \alpha_n : n < \omega \}$ is in $N$. 
By elementarity, $\sup(a \cap \theta) \in N \cap \theta$, 
so we can fix $k$ 
such that $\sup(a \cap \theta) < \alpha_k$. 
Then $a \cap \{ \alpha_n : n < \omega \} \subseteq 
a \cap \{ \alpha_n : n < k \}$. 
So $a \cap \{ \alpha_n : n < \omega \}$ is a finite subset of $N$, 
and hence is in $N$.

Since the set $\{ \alpha_n : n < \omega \}$ is countably approximated by $N$, 
and is a bounded subset of $N \cap On$, 
there exists $e \in N$ such that 
$\{ \alpha_n : n < \omega \} = N \cap e$. 
We claim that $e = \{ \alpha_n : n < \omega \}$. 
This is a contradiction, for then $\sup(e) = \sup(N \cap \theta)$ would be 
in $N$ by elementarity, which is impossible.

As $N \cap e \subseteq \theta$, by elementarity it follows 
that $e \subseteq \theta$. 
Also, for all $\beta \in N \cap \theta$, $N \cap e \cap \beta$ is finite; 
therefore by elementarity, $e \cap \beta$ must be finite. 
So $N$ models that every proper initial segment of $e$ is finite, and hence 
$e$ is at most countable. 
Therefore $e \subseteq N$, so $e = N \cap e = \{ \alpha_n : n < \omega \}$.
\end{proof}

\begin{lemma}
Let $N$ be in $P_{\omega_2}(H(\chi))$, where 
$\chi \ge \omega_2$ is a cardinal, and 
assume that $N \prec H(\chi)$. 
Let $T \in N$ be a tree with height and size $\omega_1$. 
Suppose that $W$ is an outer model of $V$ with 
$\omega_1^V = \omega_1^W$, and $N$ is $\omega_1$-guessing in $W$. 
Then every branch of $T$ in $W$ with length $\omega_1$ is in $N$.
\end{lemma}

\begin{proof}
Since $N$ is $\omega_1$-guessing in $W$, it is easy to check 
that $N$ is $\omega_1$-guessing in $V$. 
So by Lemma 2.3, $\omega_1 \subseteq N$.

Without loss of generality, assume that the underlying set of $T$ 
is $\omega_1$. 
Let $b$ be a branch of $T$ of length $\omega_1$ in $W$. 
Then $b$ is a bounded subset of $N \cap On$. 
We claim that $b$ is 
countably approximated by $N$ in $W$. 
Let $a \in N$ be countable. 
Since $b$ has length $\omega_1$, we can fix $y \in b$ such that 
$a \cap b \subseteq \{ x \in T : x <_T y \}$. 
Then $a \cap b = a \cap \{ x \in T : x <_T y \}$. 
Now $N \prec H(\chi)$ in $V$, so the set 
$\{ x \in T : x <_T y \}$ is in $N$, and hence the set 
$a \cap b$ is in $N$.

As $N$ is $\omega_1$-guessing in $W$, there is 
$e \in N$ such that $N \cap e = b$. 
Since $N \cap e \subseteq T$, it follows 
by elementarity that $e \subseteq T$. 
But the underlying set of $T$ is $\omega_1$, which is a subset of $N$. 
So $e \subseteq N$. 
Hence $e = e \cap N = b$. 
So $e = b$, and $b \in N$.
\end{proof}

\begin{definition}
For a cardinal $\theta \ge \omega_2$, 
let $\textsf{GMP}(\theta)$ be the statement that there exist 
stationarily many sets $N \in P_{\omega_2}(H(\theta))$ such that 
$N$ is $\omega_1$-guessing. 
Let $\textsf{GMP}$ be the statement that $\textsf{GMP}(\theta)$ holds, 
for all cardinals $\theta \ge \omega_2$.\footnote{$\textsf{GMP}$ stands 
for \emph{guessing model principle}.}
\end{definition}

It is easy to see that if $\omega_2 \le \theta_0 < \theta_1$ are cardinals, 
$N \in P_{\omega_2}(H(\theta_1))$ is $\omega_1$-guessing, 
$N \prec H(\theta)$, and $\theta_0 \in N$, 
then $N \cap H(\theta_0)$ is $\omega_1$-guessing. 
In particular, $\textsf{GMP}(\theta_1)$ implies $\textsf{GMP}(\theta_0)$.

Weiss \cite{weiss} introduced the principle $\textsf{ISP}(\omega_2)$, 
and Viale-Weiss \cite{vialeweiss} 
showed that it is equivalent to what we are calling $\textsf{GMP}$. 
They proved that $\textsf{ISP}(\omega_2)$ follows from $\textsf{PFA}$, 
and that $\textsf{ISP}(\omega_2)$ 
implies some of the consequence of $\textsf{PFA}$, 
such as the failure of 
square principles.

The original formulation of $\textsf{ISP}(\omega_2)$ involves completely 
different concepts than guessing models, namely ineffable branches in 
slender $P_\kappa(\lambda)$-lists. 
Some of the consequences of $\textsf{ISP}(\omega_2)$ were derived in 
\cite{vialeweiss} and \cite{weiss} using these different concepts. 
We offer new proofs of two of the most important consequences of 
$\textsf{ISP}(\omega_2)$ using arguments involving guessing models, namely, 
the failure of the approachability property on $\omega_1$, and the 
nonexistence of $\omega_2$-Aronszajn trees. 
We also derive a new consequence, namely the nonexistence of 
weak $\omega_1$-Kurepa trees.

For the original proof of the next result, see 
\cite[Corollary 4.9]{vialeweiss}.

\begin{proposition}
$\textsf{GMP}(\omega_3)$ implies $\neg \textsf{AP}_{\omega_1}$.
\end{proposition}

\begin{proof}
Suppose for a contradiction that the approachability property 
$\textsf{AP}_{\omega_1}$ holds. 
So there exists a sequence 
$\vec a = \langle a_i : i < \omega_2 \rangle$ of countable subsets of 
$\omega_2$, a club $C \subseteq \omega_2$, and 
a sequence $\vec c = \langle c_\alpha : \alpha \in C \cap \cof(\omega_1) \rangle$ 
such that for all $\alpha \in C \cap \cof(\omega_1)$, 
$c_\alpha$ is a club subset of $\alpha$ with order type $\omega_1$, 
and for all $\beta < \alpha$, there is $i < \alpha$ such that 
$c_\alpha \cap \beta = a_i$.

Fix $N$ in $P_{\omega_2}(H(\omega_3))$ such that 
$N \prec H(\omega_3)$, $\vec a$, $C$, and $\vec c$ are in $N$, 
and $N$ is $\omega_1$-guessing. 
Since $C \in N$, $N \cap \omega_2 \in C$. 
As $\omega_2 \in N$, it follows that 
$\cf(N \cap \omega_2) = \omega_1$ by Lemma 2.3. 
So $N \cap \omega_2 \in C \cap \cof(\omega_1)$. 

Let $\alpha := N \cap \omega_2$. 
We claim that $c_\alpha$ is countably approximated by $N$. 
Let $a \in N$ be countable. 
Then for some $\beta < \alpha$, $c_\alpha \cap a \subseteq \beta$. 
Fix $i < \alpha$ such that $c_\alpha \cap \beta = a_i$. 
Then $i \in N$, and hence $c_\alpha \cap \beta = a_i \in N$. 
So $c_\alpha \cap a = c_\alpha \cap \beta \cap a = 
a_i \cap a$, which is in $N$ since $a$ and $a_i$ are in $N$.

As $N$ is $\omega_1$-guessing, we can 
fix $e \in N$ such that $e \cap N = c_\alpha$. 
Since $N \cap e \subseteq \omega_2$, it follows that 
$e \subseteq \omega_2$ by elementarity. 
And as $N \cap e = c_\alpha$ is cofinal in $N \cap \omega_2$, 
$e$ is cofinal in $\omega_2$ by elementarity. 
Therefore $e$ has order type $\omega_2$. 
By elementarity, fix $\gamma \in N \cap e$ such that 
$e \cap \gamma$ has order type $\omega_1$. 
Since $\omega_1 \subseteq N$, $e \cap \gamma \subseteq N$, 
so $e \cap \gamma = e \cap N \cap \gamma = c_\alpha \cap \gamma$. 
So $c_\alpha \cap \gamma$ has order type $\omega_1$, which contradicts 
that $c_\alpha$ has order type $\omega_1$ and $\gamma < \alpha$.
\end{proof}

The next result was originally proven in \cite[Section 2]{weiss}.

\begin{thm}
$\textsf{GMP}(\omega_3)$ implies that there does not exist 
an $\omega_2$-Aronszajn tree.
\end{thm}

\begin{proof}
Let $T$ be a tree of height $\omega_2$, all of whose levels have 
cardinality less than $\omega_2$. 
We will prove that there is a branch of $T$ with order type $\omega_2$. 
Without loss of generality, assume that $T$ has underlying set $\omega_2$, 
so that $T \in H(\omega_3)$. 
Fix $N$ in $P_{\omega_2}(H(\omega_3))$ such that 
$N \prec H(\omega_3)$, $T \in N$, and $N$ is $\omega_1$-guessing.

Let $\alpha := N \cap \omega_2$. 
For all $\beta < \alpha$, $T_\beta$ is in $N$ by elementarity, and since 
$T_\beta$ has size at most $\omega_1$, 
$T_\beta \subseteq N$. 
It follows that $T \restriction \alpha \subseteq N$. 

Fix a node $y$ on level $\alpha$ of $T$. 
Let $d := \{ x \in T : x <_T y \}$. 
Then $d$ is a bounded subset of $N \cap On$. 
We claim that $d$ is countably approximated by $N$. 
Let $a \in N$ be countable. 
Since $\alpha$ has cofinality $\omega_1$ by Lemma 2.3, 
there is $y^* <_T y$ such that 
$a \cap d = a \cap \{ x \in T : x <_T y^* \}$. 
As $a$ and $y^*$ are in $N$, so is $a \cap d$.

Since $N$ is $\omega_1$-guessing, we can fix $e \in N$ such that 
$e \cap N = d$. 
We claim that $e$ is a branch of $T$ with length $\omega_2$, which 
completes the proof. 
First, if $x$ and $y$ are in $e \cap N$, then 
$x$ and $y$ are in $d$, and since $d$ is a branch, either 
$x \le_T y$ or $y <_T x$. 
By elementarity, $e$ is a chain in $T$. 
Similarly, if $x \in e \cap N$, $x_0 \in N \cap T$, and 
$x_0 <_T x$, then 
$x_0 \in d$ and hence $x_0 \in d \subseteq e$. 
By elementarity, $e$ is closed downwards. 
If $e$ does not have order type $\omega_2$, then by elementarity there is 
$\beta < N \cap \omega_2 = \alpha$ such that 
$e \subseteq T \restriction \beta$. 
But the node in $d$ at level $\beta$ is in $e$, which is a contradiction.
\end{proof}

\begin{thm}
$\textsf{GMP}(\omega_2)$ implies that there does not exist a 
weak $\omega_1$-Kurepa tree.
\end{thm}

\begin{proof}
Suppose for a contradiction that $T$ is a weak $\omega_1$-Kurepa tree. 
Without loss of generality, assume that the underlying set of $T$ 
is $\omega_1$. 
Then $T \in H(\omega_2)$. 
Fix a set $N \in P_{\omega_2}(H(\omega_2))$ such that 
$N \prec H(\omega_2)$, $T \in N$, and $N$ is $\omega_1$-guessing. 
By Lemma 2.4, every branch of $T$ with length $\omega_1$ is in $N$. 
But this is impossible, since $N$ has size $\omega_1$ and 
$T$ has more than $\omega_1$ many branches 
of length $\omega_1$.
\end{proof}

Note that if $\textsf{CH}$ holds, then there exists a 
weak $\omega_1$-Kurepa tree, namely the tree of 
functions in $2^{<\omega_1}$. 
Hence $\textsf{GMP}(\omega_2)$ implies that $\textsf{CH}$ fails. 
Viale-Weiss \cite{vialeweiss} asked whether $\textsf{GMP}$ implies 
that $2^\omega$ is equal to $\omega_2$. 
This question was settled in \cite{jk26}, where we showed that 
$\textsf{GMP}$ is consistent with $2^\omega$ being equal to 
any given cardinal $\lambda \ge \omega_2$ of uncountable cofinality.

Let us make an additional observation about the model constructed in \cite{jk26}. 
Recall that the \emph{pseudo-intersection number} $\mathfrak{p}$ is the least 
size of a collection $X$ of infinite subsets of $\omega$, closed under finite 
intersections, for which there is no set $b$ such that 
$b \setminus a$ is finite for all $a \in X$. 
Viale \cite[Lemma 4.2]{viale} 
proved that under the assumption that $\omega_1 < \mathfrak{p}$, 
if $\chi \ge \omega_2$ is a regular cardinal, 
$N \in P_{\omega_2}(H(\chi))$, $N \prec H(\chi)$, 
and $N$ is $\omega_1$-guessing, then $N$ is internally unbounded.

Viale \cite[Remark 4.3]{viale} asked whether it is consistent that there 
are stationarily many 
$\omega_1$-guessing models in a model where $\mathfrak{p} = \omega_1$. 
We point out that in the model constructed in \cite{jk26}, 
\textsf{GMP} holds and $\mathfrak{p} = \omega_1$, which settles this 
question. 
Namely, the model of \cite{jk26} 
is obtained by forcing with a forcing poset of the form 
$\p * \textrm{Add}(\omega,\lambda)$, where $\lambda > \omega_1$. 
But by \cite[Section 11.3]{blass}, any model obtained by forcing with 
$\textrm{Add}(\omega,\lambda)$, where $\lambda \ge \omega_1$, 
satisfies that $\mathfrak{p} = \omega_1$.

\section{Indestructibly Guessing Models and $\mathsf{IGMP}$}

We introduce a stronger form of $\omega_1$-guessing, 
and with it a new principle.

\begin{definition}
A set $N$ is \emph{indestructibly $\omega_1$-guessing} if 
for any forcing poset $\p$ which preserves $\omega_1$, 
$\p$ forces that $N$ is $\omega_1$-guessing.
\end{definition}

\begin{definition}
For a cardinal $\theta \ge \omega_2$, let 
$\textsf{IGMP}(\theta)$ be the statement that there exist 
stationarily many sets $N \in P_{\omega_2}(H(\theta))$ such that 
$N$ is indestructibly $\omega_1$-guessing. 
Let $\textsf{IGMP}$ be the statement that 
$\textsf{IGMP}(\theta)$ holds, for all cardinals $\theta \ge \omega_2$.\footnote{$\textsf{IGMP}$ 
stands for \emph{indestructibly guessing model principle}.}
\end{definition}

It is easy to see that if $\omega_2 \le \theta_0 < \theta_1$ are cardinals, 
$N \in P_{\omega_2}(H(\theta_1))$ is indestructibly $\omega_1$-guessing, 
$N \prec H(\theta_1)$, and $\theta_0 \in N$, then $N \cap H(\theta_0)$ is 
indestructibly $\omega_1$-guessing. 
In particular, $\textsf{IGMP}(\theta_1)$ implies $\textsf{IGMP}(\theta_0)$.

We will prove in Section 4 that $\textsf{PFA}$ implies 
$\textsf{IGMP}$.\footnote{The original proof of 
\cite[Section 4]{vialeweiss} that 
$\textsf{PFA}$ implies $\textsf{ISP}(\omega_2)$ implicitly shows that 
$\textsf{PFA}$ implies $\textsf{IGMP}$, although they did not 
formulate this principle.}

It turns out that indestructibly $\omega_1$-guessing models are 
internally unbounded. 
The proof is a variation of the proof of \cite[Lemma 4.2]{viale} that 
if $\mathfrak{p} > \omega_1$ and $N$ is $\omega_1$-guessing, then 
$N$ is internally unbounded.

\begin{proposition}
Let $\theta \ge \omega_2$ be a cardinal. 
Suppose that $N \in P_{\omega_2}(H(\theta))$, 
$N \prec H(\theta)$, $\cf(\sup(N \cap \theta)) = \omega_1$, 
and $N$ is indestructibly $\omega_1$-guessing. 
Then $N$ is internally unbounded.
\end{proposition}

\begin{proof}
Suppose for a contradiction that $x$ is a countable subset of $N$ 
which is not covered by any set in $P_{\omega_1}(N) \cap N$. 
Without loss of generality, $x$ is a set of ordinals. 
Since $\cf(N \cap On) = \cf(N \cap \theta) = \omega_1$, 
$x$ is bounded in $\sup(N \cap On)$. 
Then easily $F = \{ x \setminus y : y \in P_{\omega_1}(N) \cap N \}$ 
is a collection of infinite subsets of $x$ which is closed under finite intersections.

Let $\p$ be the $\omega_1$-c.c.\ Mathias forcing for 
adding a pseudointersection to $F$, that is, a subset of $x$ which 
is almost contained modulo finite in every member of $F$ 
(see, for example, \cite[Theorem 7.7]{blass}). 
Let $G$ be a $V$-generic filter on $\p$. 
Let $b$ be the pseudointersection given by $G$. 
Then $b \notin V$.

We claim that $b$ is countably approximated by $N$. 
Let $a$ be a countable set in $N$. 
Then $a \in P_{\omega_1}(N) \cap N$, so $x \setminus a \in F$. 
Hence $b \setminus (x \setminus a) = b \cap a$ is finite. 
As $b \cap a$ is a finite subset of $N$, it is in $N$. 
Since $N$ is indestructibly $\omega_1$-guessing and $\p$ 
is $\omega_1$-c.c., $N$ is $\omega_1$-guessing in $V[G]$. 
As $b$ is a bounded subset of $N \cap On$ which is 
countably approximated by $N$, we can 
fix $e \in N$ such that $e \cap N = b$. 
Note that $e$ is countable, for otherwise by elementarity 
$e \cap N$ would be uncountable, contradicting that $e \cap N = b$ 
and $b$ is countable. 
Therefore $e \subseteq N$, so $e = e \cap N = b$. 
But this is impossible since $e \in V$ and $b \notin V$.
\end{proof}

\begin{corollary}
Let $\theta \ge \omega_2$ be a cardinal. 
Then $\textsf{IGMP}(\theta^+)$ implies that there are stationarily 
many $N \in P_{\omega_2}(H(\theta))$ such that 
$N$ is internally unbounded and indestructibly $\omega_1$-guessing.
\end{corollary}

\begin{proof}
If $N \in P_{\omega_2}(H(\theta^+))$, $N \prec H(\theta^+)$, and 
$N$ is $\omega_1$-guessing, then 
$\sup(N \cap \theta)$ has cofinality $\omega_1$ by Lemma 2.3. 
So if $N$ is indestructibly $\omega_1$-guessing, then by the comments 
after Definition 3.2 and Proposition 3.3, 
$N \cap H(\theta)$ is an elementary substructure of 
$H(\theta)$, is indestructibly $\omega_1$-guessing, and is 
internally unbounded.
\end{proof}

\begin{corollary}
$\textsf{IGMP}$ implies $\textsf{SCH}$.
\end{corollary}

\begin{proof}
By \cite[Theorem 7.9]{viale}, $\textsf{SCH}$ holds provided that for 
all regular $\theta \ge \omega_2$, there are stationarily many 
$N \in P_{\omega_2}(H(\theta))$ such that $N$ is internally unbounded 
and $\omega_1$-guessing. 
This statement hold under $\textsf{IGMP}$ by Corollary 3.4.
\end{proof}

Now we move towards proving that $\textsf{IGMP}$ implies the 
Suslin hypothesis.

\begin{thm}
Assume $\textsf{IGMP}(\omega_2)$. 
Let $T$ be a tree with height and size $\omega_1$. 
Assume that $\p$ is a forcing poset which preserves $\omega_1$. 
Then $\p$ does not add any new branches of length $\omega_1$ to $T$.
\end{thm}

\begin{proof}
Without loss of generality, assume that the underlying set of $T$ is $\omega_1$. 
Then $T \in H(\omega_2)$. 
By $\textsf{IGMP}(\omega_2)$, we can fix $N \in P_{\omega_2}(H(\omega_2))$ 
such that $N \prec H(\omega_2)$, $T \in N$, and $N$ is 
indestructibly $\omega_1$-guessing. 
By Lemma 2.4, every branch of $T$ with length $\omega_1$ 
in a generic extension by $\p$ is in $N$, and hence in $V$.
\end{proof}

Let us say that a tree $T$ is \emph{nontrivial} if 
(1) for all $t \in T$, there are incomparable $u, v$ in $T$ with 
$t \le_T u,v$, and (2) 
for all $t \in T$ and for all $\alpha$ less than the height of $T$, 
there is $u \in T$ with height at least $\alpha$ such that $t \le_T u$.

Suppose that $T$ is a nontrivial tree. 
Define $\p_T$ as the forcing poset whose 
conditions are nodes in $T$, ordered by 
$t \le_{\p_T} s$ iff $s \le_T t$. 
The assumption of $T$ being nontrivial implies that the forcing poset $\p_T$ 
adds a branch of $T$ which is not in the ground model with length equal to 
the height of $T$.

\begin{thm}
$\textsf{IGMP}(\omega_2)$ implies that for any nontrivial tree 
$T$ with height and size $\omega_1$, $\p_T$ collapses $\omega_1$.
\end{thm}

\begin{proof}
Let $G$ be a $V$-generic filter on $\p_T$. 
Then in $V[G]$, $G$ is a branch of $T$ which is not in $V$. 
By Theorem 3.6, $\omega_1^V$ is not equal to $\omega_1^{V[G]}$.
\end{proof}

We note that the conclusion of Theorem 3.7 was previously known to follow from 
$\textsf{PFA}$ (\cite[Section 7]{baumgartner}). 
Namely, under $\textsf{PFA}$, every tree with height and size $\omega_1$ 
is special (see Definition 4.1 below). 
And adding a new branch of length $\omega_1$ 
to a special tree by forcing will collapse 
$\omega_1$ (see Proposition 4.3 below).

Recall that if $T$ is an $\omega_1$-Suslin tree, then $T$ is a tree 
with height and size $\omega_1$, and $\p_T$ is $\omega_1$-c.c. 
In particular, $\p_T$ preserves $\omega_1$. 
Moreover, if there exists an $\omega_1$-Suslin tree, then there exists 
a nontrivial $\omega_1$-Suslin tree.

\begin{corollary}
$\textsf{IGMP}(\omega_2)$ implies that there does not 
exist an $\omega_1$-Suslin tree, so the Suslin hypothesis holds.
\end{corollary}

\begin{proof}
Immediate from Theorem 3.7.
\end{proof}

On the other hand, the Suslin hypothesis is consistent with 
the principle $\textsf{GMP}$. 
For example, the model of $\textsf{GMP}$ constructed in \cite{jk26} 
is a generic extension by a forcing poset of the form 
$\p * \textrm{Add}(\omega,\lambda)$. 
But Shelah \cite{shelah} proved that Cohen forcing 
$\textrm{Add}(\omega)$ adds an 
$\omega_1$-Suslin tree, so there exists an $\omega_1$-Suslin tree 
in this model.

\begin{thm}
Assume that $2^\omega \le \omega_2$, and there exist cofinally many 
sets in $P_{\omega_2}(H(\omega_2))$ which are 
indestructibly $\omega_1$-guessing. 
Suppose that $W$ is a generic extension of $V$, and $W$ contains a subset 
of $\omega_1$ which is not in $V$. 
Then either $W$ contains a real which is not in $V$, 
or $\omega_2^V$ is not a cardinal in $W$.
\end{thm}

\begin{proof}
Since $2^\omega \le \omega_2$, $H(\omega_1)$ has size at most $\omega_2$. 
So we can inductively define a $\subseteq$-increasing sequence 
$\langle N_i : i < \omega_2 \rangle$, 
whose union contains $H(\omega_1)$, 
such that for each $i < \omega_2$, 
$N_i$ is in $P_{\omega_2}(H(\omega_2))$, 
$\omega_1 + 1 \subseteq N_i$, and $N_i$ is indestructibly $\omega_1$-guessing. 

Suppose that $W \setminus V$ does not contain a real, and 
we will show that $\omega_2^V$ is not a cardinal in $W$. 
Since every subset of $\omega$ in $W$ is in $V$, it follows that 
$\omega_1^V = \omega_1^W$, and 
$W \setminus V$ contains no bounded subset of $\omega_1$. 
Fix $b$ which is a subset of $\omega_1$ in $W \setminus V$. 
Then every proper initial segment of $b$ is in $V$, and hence in $H(\omega_1)^V$. 
So we can fix, for each $\alpha < \omega_1$, the least ordinal 
$i_\alpha < \omega_2^V$ such that $b \cap \alpha \in N_{i_\alpha}$. 
Note that the sequence $\langle i_\alpha : \alpha < \omega_1 \rangle$ is in $W$.

We claim that this sequence is unbounded in $\omega_2^V$, which implies 
that $\omega_2^V$ is not a cardinal in $W$. 
Otherwise there is $\delta < \omega_2^V$ such that 
$i_\alpha < \delta$ for all $\alpha < \omega_1$. 
Since the sequence $\langle N_\alpha : \alpha < \omega_2 \rangle$ is 
$\subseteq$-increasing, it follows that 
for all $\alpha < \omega_1$, $b \cap \alpha \in N_\delta$. 
As $\omega_1 + 1 \subseteq N_\delta$, 
$b$ is a bounded subset of $N_\delta \cap On$. 

For any countable set $a \in N_\delta$, 
$a \cap b = a \cap (b \cap \alpha)$ for some $\alpha < \omega_1$, 
and hence $a \cap b \in N_\delta$. 
So $b$ is countably approximated by $N_\delta$ in $W$. 
Since $N_\delta$ is indestructibly $\omega_1$-guessing in $V$, $N_\delta$ 
is $\omega_1$-guessing in $W$. 
So there is $e \in N_\delta$ such that $b = e \cap N_\delta$. 
But then $e$ and $N_\delta$ are in $V$, and hence $b \in V$, 
which is a contradiction.
\end{proof}

We note that the conclusion of Theorem 3.9 
was previously shown to follow from $\textsf{PFA}$ 
by {Todor\v cevi\' c} \cite[Theorem 2]{todorcevic}.

\section{Trees and Guessing Models}

In this section we review some ideas of Baumgartner and Viale-Weiss 
concerning trees and guessing models. 
The main result of this section is Corollary 4.5, which gives 
a sufficient condition under which $\textsf{IGMP}$ holds. 
We also observe that $\textsf{PFA}$ implies $\textsf{IGMP}$.

The next definition is due to Baumgartner \cite[Section 7]{baumgartner}.

\begin{definition}
Let $T$ be a tree. 
We say that $T$ is \emph{special} if there exists a function 
$f : T \to \omega$ such that whenever $s$, $t$, and $u$ are in $T$ and 
$f(s) = f(t) = f(u)$, if $s <_T t$ and $s <_T u$, 
then $t$ and $u$ are comparable in $T$.
\end{definition}

Recall that for a tree $T$, 
a function $g : T \to \omega$ is a \emph{specializing function} if 
for all $s, t \in T$, if $s <_T t$ then $g(s) \ne g(t)$. 
Clearly if $T$ has a specializing function, then $T$ has no branches 
of length $\omega_1$.
Baumgartner \cite[Theorem 7.3]{baumgartner} proved that 
for a tree $T$ of height $\omega_1$ which has no branches of length 
$\omega_1$, $T$ is special in the sense of Definition 4.1 iff 
$T$ has a specializing function.

The next theorem appears as Theorem 7.5 of \cite{baumgartner}.

\begin{thm}
Assume that every tree of height and size $\omega_1$ which has 
no branches of length $\omega_1$ is special. 
Then every tree of height and size $\omega_1$ which has at 
most $\omega_1$ many branches of length $\omega_1$ is special.
\end{thm}

\begin{proposition}
Suppose that $T$ is a tree with height $\omega_1$ which is special. 
Then whenever $W$ is an outer model of $V$ with the same $\omega_1$, 
any branch of $T$ in $W$ of length $\omega_1$ is in $V$.
\end{proposition}

\begin{proof}
The proof follows easily from ideas of Baumgartner 
\cite[Section 7]{baumgartner}. 
Fix a function $f : T \to \omega$ such that whenever $s, t, u$ are in $T$ and 
$f(s) = f(t) = f(u)$, if $s <_T t$ and $s <_T u$ then 
$t$ and $u$ are comparable. 
Let $W$ be an outer model with $\omega_1^V = \omega_1^W$, 
and suppose that $b$ is a branch of $T$ with 
length $\omega_1$ in $W$. 
We will show that $b \in V$. 

Since $f \restriction b$ is a function from a set of size $\omega_1$ 
into $\omega$, we can fix $n < \omega$ such that 
the set $\{ t \in b : f(t) = n \}$ has size $\omega_1$.
Fix $s$ in $b$ such that $f(s) = n$. 
Then the set 
$$
X := \{ t \in b : s <_T t, \ f(t) = n \}
$$
is uncountable. 
But $X$ is a subset of the set 
$$
Y := \{ t \in T : s <_T t, \ f(t) = n \},
$$
and hence $Y$ is uncountable. 
Note that $Y$ is in $V$. 

The set $Y$ is an uncountable chain. 
For if $t$ and $u$ are in $Y$, then 
$s <_T t$, $s <_T u$, and $f(s) = f(t) = f(u) = n$. 
Since $T$ is special, it follows that $t$ and $u$ are comparable. 
Let $c$ be the downwards closure of $Y$. 
Then $c$ is a branch of $T$ in $V$ with length $\omega_1$, 
and $c \in V$. 
There are cofinally many nodes above $s$ in $b$ which 
take value $n$ under $f$, and any such node is in $c$. 
So $b = c$, and therefore $b \in V$.
\end{proof}

We now establish a connection between special trees 
and indestructibly $\omega_1$-guessing models. 
This connection involves constructing a tree 
from a guessing model; a similar construction was done 
previously in \cite[Lemma 4.6]{vialeweiss}.

\begin{proposition}
Let $\theta \ge \omega_2$ be a cardinal. 
Suppose that $N$ is in $P_{\omega_2}(H(\theta))$, 
$N \prec H(\theta)$, and $N$ is internally unbounded 
and $\omega_1$-guessing. 
Then there exists a tree $T$ of height and size $\omega_1$ 
which has $\omega_1$ many branches of length $\omega_1$ 
such that $T$ being special implies 
that $N$ is indestructibly $\omega_1$-guessing.
\end{proposition}

\begin{proof}
Since $N$ is internally unbounded, we can 
fix a $\subseteq$-increasing sequence $\langle N_i : i < \omega_1 \rangle$ 
with union equal to $N$ such that for all $i < \omega_1$, 
$N_i \in P_{\omega_1}(N) \cap N$. 

Fix an uncountable ordinal $\delta$ in $N$, 
and we will define a tree $T_\delta$. 
The desired tree $T$ will then be the direct sum over all such trees $T_\delta$. 
The underlying set of $T_\delta$ consists of all pairs in $N$ of the 
form $(i,f)$, where $i < \omega_1$ and 
$f : N_i \cap \delta \to 2$. 
For $(i,f)$ and $(j,g)$ in $T_\delta$, let $(i,f) <_{T_\delta} (j,g)$ if 
$i < j$ and $f \subseteq g$. 
Note that $T_\delta$ is a tree with height and size $\omega_1$.

We claim that $T_\delta$ has at most $\omega_1$ many branches of length $\omega_1$. 
Consider a branch $b$ of length $\omega_1$, and let 
$F_b := \bigcup \{ f : \exists i \ (i,f) \in b \}$. 
Note that $F_b$ is a function with domain equal to $N \cap \delta$, and 
$b = \{ (i,F_b \restriction N_i) : i < \omega_1 \}$. 
Since $b \subseteq T_\delta \subseteq N$, 
we have that for all $i < \omega_1$, 
$F_b \restriction N_i \in N$.

Let $A_b := \{ \alpha \in N \cap \delta : F_b(\alpha) = 1 \}$. 
We claim that $A_b$ is $N$-guessed. 
As $N$ is $\omega_1$-guessing, it is enough to show that $A_b$ is 
countably approximated by $N$. 
So let $a$ be a countable set in $N$. 
Then for some $i < \omega_1$, $a \subseteq N_i$. 
Therefore $a \cap A_b = a \cap A_b \cap N_i$. 
Now $A_b \cap N_i = \{ \alpha \in N_i \cap \delta : F_b(\alpha) = 1 \}$, which is 
definable in $N$ from the parameters $N_i$, $\delta$, and 
$F_b \restriction N_i$. 
It follows that $A_b \cap N_i$ is in $N$. 
Since $a$ is also in $N$, $a \cap A_b = a \cap A_b \cap N_i$ is in $N$. 
Since $N$ is $\omega_1$-guessing, we can fix $e_b$ in $N$ such that 
$A_b = e_b \cap N$.

Suppose that $b$ and $c$ are distinct branches of $T$ with length $\omega_1$. 
Then easily the functions $F_b$ and $F_c$ are different, and therefore 
the sets $A_b$ and $A_c$ are distinct. 
Since $A_b = e_b \cap N$ and $A_c = e_c \cap N$, it follows that 
$e_b$ and $e_c$ are distinct. 
Thus the map $b \mapsto e_b$ from the set of branches of $T_\delta$ 
with length $\omega_1$ into $N$ is injective. 
Since $N$ has size $\omega_1$, it follows that $T_\delta$ has no more than 
$\omega_1$ many branches of length $\omega_1$. 
This completes the analysis of $T_\delta$.

Let $T$ be the disjoint sum of the trees $T_\delta$, for $\delta$ an 
uncountable ordinal in $N$. 
In other words, the underlying set of $T$ consists of pairs of the form 
$(\delta,t)$, where $\delta$ is an uncountable ordinal in $N$ and 
$t \in T_\delta$. 
And the order on $T$ is given by letting 
$(\delta_1,t_1) <_T (\delta_2,t_t)$ iff 
$\delta_1 = \delta_2$ and $t_1 <_{T_\delta} t_2$. 
Then $T$ is a tree of height and size $\omega_1$ which has 
at most $\omega_1$ many branches of length $\omega_1$.

Suppose that $T$ is special, 
and we will show that $N$ is indestructibly $\omega_1$-guessing. 
Let $W$ be an outer model of $V$ with $\omega_1^V = \omega_1^W$. 
Assume that $d$ is a bounded subset of $N \cap On$ in $W$ which is 
countably approximated by $N$. 
We will show that $d$ is $N$-guessed. 
Fix an uncountable ordinal $\delta$ in $N$ 
such that $d \subseteq \delta$.

Let $h : N \cap \delta \to 2$ be the characteristic function of $d$, in other words, 
$h(\alpha) = 1$ if $\alpha \in d$, and $h(\alpha) = 0$ otherwise. 
We claim that for all $i < \omega_1$, $(i,h \restriction N_i)$ is in 
$T_\delta$. 
It suffices to show that $h \restriction N_i$ is in $N$. 
Since $N_i \in N$, $N_i$ is countable, 
and $d$ is countably approximated by $N$, 
it follows that $d \cap N_i \in N$. 
But $h \restriction N_i$ is the characteristic function of $d \cap N_i$, 
so $h \restriction N_i \in N$. 
Hence $(i,h \restriction N_i)$ is in $T_\delta$.

It follows that $b := \{ (i,h \restriction N_i) : i < \omega_1 \}$ is a 
branch of $T_\delta$, and hence of $T$, with length $\omega_1$. 
Since $T$ is special, any branch of $T$ in $W$ with length $\omega_1$ 
is in $V$ by Proposition 4.3. 
It follows that $b \in V$. 
Therefore $h \in V$, and hence $d \in V$. 
Since $d$ is countably approximated by $N$ in $W$, it is 
also countably approximated by $N$ in $V$. 
As $N$ is $\omega_1$-guessing in $V$, $d$ is $N$-guessed.
\end{proof}

\begin{corollary}
Suppose that every tree of height and size $\omega_1$ which has 
no branches of length $\omega_1$ is special. 
Assume that for all sufficiently large 
regular cardinals $\theta \ge \omega_2$, there are stationarily many 
sets $N$ in $P_{\omega_2}(H(\theta))$ such that $N$ is 
internally unbounded and $\omega_1$-guessing. 
Then $\textsf{IGMP}$ holds.
\end{corollary}

\begin{proof}
By Theorem 4.2, every tree of height and size $\omega_1$ which has 
at most $\omega_1$ many branches of length $\omega_1$ is special. 
By the comments after Definition 3.2, it suffices to show that for all 
sufficiently large regular cardinals $\theta \ge \omega_2$, 
$\textsf{IGMP}(\theta)$ holds. 
By assumption, for all sufficiently large regular cardinals 
$\theta \ge \omega_2$, there 
are stationarily many $N \in P_{\omega_2}(H(\theta))$ such that 
$N \prec H(\theta)$, $N$ is internally unbounded, and $N$ is 
$\omega_1$-guessing. 
By Proposition 4.4, for any such $N$ there exists a tree $T$ with height 
and size $\omega_1$ which has at most $\omega_1$ many branches 
of length $\omega_1$ such that 
if $T$ is special then $N$ is indestructibly $\omega_1$-guessing. 
By our assumption about trees, $T$ is indeed special, so $N$ is indestructibly 
$\omega_1$-guessing.
\end{proof}

\begin{corollary}
$\textsf{PFA}$ implies $\textsf{IGMP}$.
\end{corollary}

\begin{proof}
By \cite[Section 4]{vialeweiss}, $\textsf{PFA}$ implies that for all 
regular cardinals $\theta \ge \omega_2$, there are stationarily many $N$ in 
$P_{\omega_2}(H(\lambda))$ which are internally unbounded 
and $\omega_1$-guessing. 
By \cite{baumgartner2}, $\textsf{MA}$, and hence $\textsf{PFA}$, implies that 
every tree of height and size $\omega_1$ which has no branches of 
length $\omega_1$ is special. 
The result now follows from Corollary 4.5.
\end{proof}

\bigskip

Corollary 3.4 and Proposition 4.4 
suggest an alternative definition of indestructibly $\omega_1$-guessing. 
Let us call an internally unbounded $\omega_1$-guessing model a 
\emph{special $\omega_1$-guessing model} if some tree as described 
in the proof of Proposition 4.4 is special. 
The argument of Proposition 4.4 shows that in that case, 
$N$ is $\omega_1$-guessing in any outer model $W$ with the 
same $\omega_1$. 
This conclusion about $N$ is apparently stronger than being indestructibly 
$\omega_1$-guessing, since the latter property is restricted to outer models 
$W$ which are generic extensions of $V$. 

Thus we could formulate another principle which asserts that there exist 
stationarily many special $\omega_1$-guessing models, and this principle 
clearly implies $\textsf{IGMP}$. 
Note that by the proof of Corollary 4.6, \textsf{PFA} implies 
this principle. 
We do not know whether the two principles are equivalent, so we leave this 
as an open question. 
They are equivalent if $\textsf{IGMP}$ implies that every tree of height 
and size $\omega_1$ with at most $\omega_1$ many branches is special, 
but that is not known.

\section{Strong Genericity and the Strongly Proper Collapse}

We now turn to developing the forcing posets which will be 
used in the consistency result of Section 8. 
In this section we review the ideas of strong genericity and strong properness, 
prove a theorem about the preservation of strong properness 
by proper forcing, 
and discuss the strongly proper collapse. 
More details on these topics can be found in \cite{jk26}.

\begin{definition}
Let $\q$ be a forcing poset, $q \in \q$, and $N$ a set. 
Then $q$ is a \emph{strongly $(N,\q)$-generic condition} 
if for any set $D$ which is a dense subset of the forcing poset $N \cap \q$, 
$D$ is predense in $\q$ below $q$.
\end{definition}
 
If $\q$ is understood from context, we say that $q$ is a 
strongly $N$-generic condition. 

\begin{lemma}
Let $\q$ be a forcing poset, $q \in \q$, and $N$ a set. 
Then $q$ is strongly $N$-generic iff there exists 
a function $r \mapsto r \restriction N$, defined on 
conditions $r \le q$, satisfying that $r \restriction N \in N \cap \q$, 
and for all $v \le r \restriction N$ in $N \cap \q$, $r$ and $v$ are compatible.
\end{lemma}

\begin{proof}
See \cite[Lemma 2.2]{jk26}.
\end{proof}

For a forcing poset $\q$, let $\lambda_\q$ denote the smallest 
cardinal $\lambda$ such that $\q \subseteq H(\lambda)$.

\begin{definition}
A forcing poset $\q$ is \emph{strongly proper on a stationary set} 
if there are 
stationarily many $N$ in $P_{\omega_1}(H(\lambda_\q))$ such that 
whenever $p \in N \cap \q$, 
there is $q \le p$ which is a strongly $N$-generic condition.
\end{definition}

Standard arguments show that being strongly proper on a stationary 
set is equivalent to the property above, where we replace $\lambda_\q$ 
with any cardinal $\theta \ge \lambda_\q$.
 
\begin{proposition}
If $\q$ is strongly proper on a stationary set, then $\q$ has the 
$\omega_1$-covering property and the $\omega_1$-approximation property.
\end{proposition}

\begin{proof}
A condition which is strongly $N$-generic is also $N$-generic in the 
sense of proper forcing. 
By standard proper forcing arguments, $\q$ 
has the $\omega_1$-covering property. 
For a proof that $\q$ has the $\omega_1$-approximation property, 
see the comments after \cite[Proposition 2.13]{jk26}.
\end{proof}

\begin{thm}
Suppose that $\q$ is strongly proper on a stationary set, and $\p$ is proper. 
Then $\p$ forces that $\q$ is strongly proper on a stationary set.
\end{thm}

\begin{proof}
Fix $\theta$ such that $\p$ forces that 
$\theta$ is a cardinal and $\theta \ge \lambda_\q$. 
Fix a $\p$-name $\dot F$ for a function from 
$(H(\theta)^{V[\dot G_\p]})^{<\omega}$ to 
$H(\theta)^{V[\dot G_\p]}$, and let $p \in \p$. 
We will find $u \le p$, and a name $\dot M$ for a countable subset of 
$H(\theta)^{V[\dot G_\p]}$ which is closed under $\dot F$, 
such that $u$ forces that 
for all $s \in \dot M \cap \q$, 
there is $t \le s$ which is strongly $(\dot M,\q)$-generic.

Let $\chi$ be a regular cardinal larger than $2^{|\p|}$ such that 
$\p$, $\q$, $\theta$, and $\dot F$ are in $H(\chi)$. 
Since $\p$ is proper and $\q$ is strongly proper on a stationary set, 
we can fix $N$ in $P_{\omega_1}(H(\chi))$ satisfying:
\begin{enumerate}
\item $N \prec (H(\chi),\in,\p,p,\q,\theta,\dot F)$;
\item for all $p_0 \in N \cap \p$, there is $q_0 \le p_0$ which is 
$(N,\p)$-generic;
\item for all $s \in N \cap \q$, there is $t \le s$ which is strongly 
$(N,\q)$-generic.
\end{enumerate}
Since $p \in N \cap \p$, by (2) we can fix $q \le p$ which is 
$(N,\p)$-generic. 
We claim that $q$ forces that 
$$
\dot M := N[\dot G_\p] \cap H(\theta)^{V[\dot G_\p]}
$$
is as required.

Since $q$ is $(N,\p)$-generic, $q$ forces that 
$N[\dot G_\p] \cap V = N$. 
By (1), $\p$ forces that 
$$
N[\dot G_\p] \prec (H(\chi)^{V[\dot G_\p]},\in,\dot F),
$$
and therefore that $N[\dot G_\p]$ is closed under $\dot F$.
Hence $\p$ forces that $\dot M$ is closed under $\dot F$.

Let $r \le q$ and $\dot s$ be given such that $r$ forces in $\p$ that 
$\dot s \in \dot M \cap \q$. 
Then $r$ forces that $\dot s \in N[\dot G_\p] \cap \q \subseteq 
N[\dot G_\p] \cap V = N$. 
So we can fix $u \le r$ and $s \in N$ such that $u$ forces that 
$\dot s = \check s$.

By (3), let $t \le s$ be strongly $(N,\q)$-generic. 
Then there exists a function $g : \{ z \in \q : z \le t \} \to N \cap \q$ 
satisfying that 
for all $z \le t$ in $\q$, if $w \le g(z)$ is in $N \cap \q$, 
then $w$ and $z$ are compatible in $\q$. 
Note that by upwards absoluteness, $g$ is forced to 
satisfy the same property in $V[\dot G_\p]$. 
But $u$ forces that $N \cap \q = N[\dot G_\p] \cap \q = \dot M \cap \q$. 
Therefore $u$ forces that $g : \{ z \in \q : z \le t \} \to \dot M \cap \q$ 
and for all 
$z \le t$ in $\q$, whenever $w \le g(z)$ is in $\dot M \cap \q$, 
then $w$ and $z$ are compatible in $\q$. 
In other words, $u$ forces that $t$ 
is strongly $(\dot M,\q)$-generic.
\end{proof}

Assume that $\kappa$ is a strongly inaccessible cardinal. 
In the proof of the consistency result in Section 8, 
we will use the forcing poset $\p$ from \cite[Section 6]{jk26}, 
which is called a strongly proper collapse. 
The forcing poset $\p$ is strongly proper, $\kappa$-c.c., has size $\kappa$, 
and collapses $\kappa$ to become $\omega_2$. 
Roughly speaking, this forcing poset consists of finite adequate sets of 
countable elementary substructures, ordered by reverse inclusion. 
A more detailed description of this forcing poset 
is beyond the scope of this paper; see \cite[Section 6]{jk26} for more details.

We will need one more technical fact about the strongly proper collapse $\p$.

\begin{proposition}
Let $\lambda \ge \kappa$ be a cardinal. 
Then $\p \times \textrm{Add}(\omega,\lambda)$ is strongly proper. 
Moreover, if $\p_0$ is any regular suborder of 
$\p \times \textrm{Add}(\omega,\lambda)$, then $\p_0$ forces that 
$(\p \times \textrm{Add}(\omega,\lambda)) / \dot G_{\p_0}$ is 
strongly proper on a stationary set.
\end{proposition}

\begin{proof}
This follows from Theorem 2.11 and Proposition 7.3 of \cite{jk26}.
\end{proof}

\section{Special iterations}

To obtain a model in which $\textsf{IGMP}$ holds, we will use a finite 
support iteration $\p$ of forcings which specialize trees of height and size 
$\omega_1$ which have no branches of length $\omega_1$. 
It was proven recently in \cite{ycc} that such an iteration has the 
$\omega_1$-approximation property.

For our purposes, we will need to know that a certain quotient of such an 
iteration has the $\omega_1$-approximation property. 
Specifically, we will have an elementary substructure $N$ 
of size $\omega_1$, and we will need to know that the regular suborder 
$N \cap \p$ forces that $\p / \dot G_{N \cap \p}$ has the 
$\omega_1$-approximation property. 
Unlike the situation in \cite{jk26}, we do not have a general result 
which implies that such a quotient has the $\omega_1$-approximation property. 
Instead, we will prove directly that $\p / \dot G_{N \cap \p}$ is forced 
by $N \cap \p$ to be forcing equivalent 
to a finite support iteration of specializing forcings. 

Let $\textrm{Fn}(\omega_1,\omega)$ denote the set of all 
finite functions whose domain is a subset of $\omega_1$ and whose 
range is a subset of $\omega$. 
Recall that if $T$ is a tree with no branches of length $\omega_1$, then 
$P(T)$ is the forcing poset described in Section 1 for adding a 
specializing function to $T$.

\begin{definition}
For an ordinal $\lambda$, let $S(\lambda)$ denote the set of all 
functions $p$, whose domain is a finite subset of $\lambda$, 
such that for all $\alpha \in \dom(p)$, 
$p(\alpha) \in \textrm{Fn}(\omega_1,\omega)$. 
Define a partial order on $S(\lambda)$ by letting $q \le p$ if 
$\dom(p) \subseteq \dom(q)$, and for all $\alpha \in \dom(p)$, 
$p(\alpha) \subseteq q(\alpha)$.
\end{definition}

Note that if $W$ is an outer model of $V$ with 
$\omega_1^V = \omega_1^W$, then 
$S(\lambda)^V = S(\lambda)^W$, and the order on $S(\lambda)$ 
is the same in $V$ and $W$.

Instead of working directly with forcing iterations, we will use 
a simpler dense suborder.

\begin{definition}
For an ordinal $\lambda$ and a set $A \subseteq \lambda$, 
a sequence $\langle \p_i : i \le \lambda \rangle$ is said to be 
a \emph{special $A$-iteration} if there a exist sequence 
$\langle \dot T_i : i \in A \rangle$ 
such that the following statements are satisfied:
\begin{enumerate}
\item for all $i \le \lambda$, $\p_i$ is a suborder of $S(i)$;
\item for all $i \in A$, $\dot T_i$ is a $\p_i$-name for a 
tree with underlying set $\omega_1$ which has no branches of length $\omega_1$;
\item $\p_0 = \{ \emptyset \}$;
\item for all $i \in A$, a set $p \in S(i+1)$ is in $\p_{i+1}$ iff 
$p \restriction i \in \p_i$ and if $i \in \dom(p)$ then 
$p \restriction i \Vdash_{\p_i} p(i) \in P(\dot T_i)$;
\item for all $i \in \lambda \setminus A$, $\p_{i+1} = \p_{i}$;
\item for all $\beta \le \lambda$ limit, a set $p \in S(\beta)$ is in $\p_\lambda$ iff 
for all $i < \beta$, $p \restriction i \in \p_i$.
\end{enumerate}
\end{definition}

Note that if $p \in \p_\lambda$, then $\dom(p) \subseteq A$. 
If $A = \lambda$, then we say that the sequence is a \emph{special iteration}. 
The partial ordering $\p_\lambda$ itself is said to be a 
special $A$-iteration.

The next two lemmas provide some basic facts about special $A$-iterations.

\begin{lemma}
Let $\langle \p_i : i \le \lambda \rangle$ be a special $A$-iteration. 
Let $i < j \le \lambda$. 
Then:
\begin{enumerate}
\item $\p_i \subseteq \p_j$;
\item if $p \in \p_j$, then $p \restriction i \in \p_i$;
\item if $p \in \p_j$ and $q \le p \restriction i$ in $\p_i$, then 
$q \cup p \restriction [i,j)$ is in $\p_j$ and is below $p$;
\item the function $p \mapsto p \restriction i$ is a projection mapping of 
$\p_j$ onto $\p_i$, and this map satisfies that $p \restriction i = p$ for 
all $p \in \p_i$, and $q \le q \restriction i$ for all $q \in \p_j$;
\item $\p_i$ is a regular suborder of $\p_j$.
\end{enumerate}
\end{lemma}

\begin{proof}
(1), (2), and (3) can be easily proven by induction. 
(4) follows from (3), and (4) implies (5) by Lemma 1.5.
\end{proof}

\begin{lemma}
Let $\langle \p_i : i \le \lambda \rangle$ be a special $A$-iteration, with 
sequence of names $\langle \dot T_i : i \in A \rangle$. 
Let $p$ and $q$ be conditions in $\p_\lambda$ satisfying that for all 
$i \in \dom(p) \cap \dom(q)$, $p(i) \subseteq q(i)$. 
Define $p + q$ as the function 
with domain equal to $\dom(p) \cup \dom(q)$, such that 
for all $i \in \dom(p + q)$, if $i \in \dom(q)$ then $(p + q)(i) = q(i)$, and if 
$i \in \dom(p) \setminus \dom(q)$, then $(p + q)(i) = p(i)$. 
Then $p + q$ is in $\p_\lambda$, and $p + q \le p, q$.
\end{lemma}

\begin{proof}
Let $r := p + q$. 
It is clear that $r$ is in $S(\lambda)$, and $r \le p, q$ in $S(\lambda)$. 
In fact, for all $i \le \lambda$, $r \restriction i \in S(i)$, and 
$r \restriction i \le p \restriction i, q \restriction i$ in $S(i)$. 
To show that $r \in \p_\lambda$, we will prove by induction that 
$r \restriction i \in \p_i$ for all $i \le \lambda$.

For $i = 0$, $r \restriction 0 = \emptyset$ is in $\p_0$ 
by Definition 6.2(3).  
If $\beta \le \lambda$ is a limit ordinal and for all $\gamma < \beta$, 
$r \restriction \gamma \in \p_\gamma$, then 
$r \in \p_\beta$ by Definition 6.2(6).

Suppose that $i = i_0 + 1$ and $r \restriction i_0 \in \p_{i_0}$. 
Then $(r \restriction i) \restriction i_0 = r \restriction i_0 \in \p_{i_0}$. 
If $i_0 \notin \dom(r)$, then $r \restriction i = r \restriction i_0 
\in \p_{i_0} \subseteq \p_i$, and we are done. 
Suppose that $i_0 \in \dom(r)$. 
Then $i_0 \in A$, and 
$r(i_0) = s(i_0)$, where $s$ is either $p$ or $q$. 
But $r \restriction i_0 \le s \restriction i_0$, and 
$s \restriction i_0 \Vdash_{\p_{i_0}} s(i_0) \in P(\dot T_{i_0})$. 
Hence $r \restriction i_0 \Vdash_{\p_{i_0}} r(i_0) = s(i_0) \in P(\dot T_{i_0})$. 
So $r \in \p_i$.
\end{proof}

The next lemma says that a special $A$-iteration is forcing equivalent to 
a finite support iteration of specializing forcings.

\begin{lemma}
Let $\langle \p_i : i \le \lambda \rangle$ be a special $A$-iteration, with 
sequence of names $\langle \dot T_i : i \in A \rangle$. 
Then there exists 
a finite support iteration 
$$
\langle \p_i^*, \dot \q_j^* : i \le \lambda, 
j < \lambda \rangle
$$
satisfying the following properties:
\begin{enumerate}

\item for all $i < \lambda$, 
the function which send $p \in \p_i$ 
to $p^* \in \p_i^*$, where 
$\dom(p^*) = \dom(p)$ and 
for all $\alpha \in \dom(p^*)$, $p^*(\alpha)$ is the canonical 
$\p_{\alpha}^*$-name for $p(\alpha)$, is an isomorphism from $\p_i$ 
into a dense suborder of the separative quotient of $\p_{i}^*$;

\item for all $i \in A$, 
$\p_{i}^*$ forces that $\dot \q_{i}^* = P(\dot T_i^*)$, where 
$\dot T_i^*$ is the canonical translation of the $\p_i$-name 
$\dot T_{i}$ to a $\p_i^*$-name using the isomorphism described 
in (1);

\item for all $j \in \lambda \setminus A$, 
$\p_j^*$ forces that $\dot \q_j = \{ \emptyset \}$ 
is the trivial forcing poset.
\end{enumerate}
\end{lemma}

\begin{proof}
The proof follows by standard arguments.
\end{proof}

\begin{corollary}
Let $\p$ be a special $A$-iteration. 
Then $\p$ is $\omega_1$-c.c.\ and has the $\omega_1$-approximation property.
\end{corollary}

\begin{proof}
By Lemma 6.5, $\p$ is forcing equivalent to a finite support iteration of forcings 
which specialize trees of height and size $\omega_1$ which have 
no branches of length $\omega_1$. 
So $\p$ is forcing equivalent to a finite support iteration 
of $Y$-c.c.\ forcing posets, and hence is $Y$-c.c. 
But any $Y$-c.c.\ forcing poset is $\omega_1$-c.c.\ and has 
the $\omega_1$-approximation property.
\end{proof}

\begin{proposition}
Let $\lambda$ be a cardinal of cofinality at least $2^{\omega_1}$. 
Then there exists a special iteration 
$\langle \p_i : i \le \lambda \rangle$ which forces that every 
tree with underlying set $\omega_1$ which has 
no branches of length $\omega_1$ is special.
\end{proposition}

\begin{proof}
We give a sketch of the proof. 
The special iteration is constructed by induction. 
We only need to specify the names $\dot T_j$ for $j < \lambda$, since 
the rest of the definition is determined by conditions 3--6 of Definition 6.2.

For $j < \lambda$, $\p_j$ will be a subset of $S(j)$, and hence will have size at most 
$\omega_1 \cdot |j|$, which is less than $\lambda$. 
Also $\p_j$ is $\omega_1$-c.c. 
So it is possible to enumerate all nice $\p_j$-names for trees with underlying set 
$\omega_1$ with no branches of length $\omega_1$ 
in order type less than or equal to $\lambda$. 
Using a bookkeeping function, we choose $\dot T_j$ to be such a name 
which was enumerated at some stage less than or equal to $j$. 
The bookkeeping function will ensure that any name that is enumerated 
will eventually be chosen as $\dot T_j$ for some $j < \lambda$. 

Given a nice $\p_\lambda$-name for a tree with underlying set $\omega_1$ 
with no branches of length $\omega_1$, 
since the cofinality of $\lambda$ is greater than $\omega_1$ and 
$\p_\lambda$ is $\omega_1$-c.c., it is easy to see that the name is actually 
a $\p_j$-name for some $j < \lambda$. 
So at some stage earlier than $\lambda$, we forced with $P(\dot T_j)$, 
and specialized the tree $\dot T_j$.
\end{proof}

\begin{corollary}
Let $\lambda$ be a cardinal with cofinality at least $2^{\omega_1}$. 
Then there exists a special iteration 
$\langle \p_i : i \le \lambda \rangle$ which forces that 
every tree with height and size $\omega_1$ which has no branches of 
length $\omega_1$ is special.
\end{corollary}

\begin{proof}
If $T$ is a tree with height and size $\omega_1$, then clearly $T$ is 
isomorphic to a tree with underlying set $\omega_1$. 
Now apply Proposition 6.7.
\end{proof}

We conclude this section by proving that a tail of a special iteration is 
forcing equivalent to a special iteration in an intermediate extension. 
The proof is elementary, but somewhat tedious; the reader should not 
feel guilty in just accepting the statement of Proposition 6.9 and skipping the proof.

Fix a special $A$-iteration 
$\langle \p_i : i \le \lambda \rangle$, with sequence of names 
$\langle \dot T_i : i \in A \rangle$. 
Let $i < j \le \lambda$, and 
suppose that $G_i$ is a $V$-generic filter on $\p_i$. 
By Lemma 1.1, if $G_{i,j}$ is a $V[G_i]$-generic filter on $\p_j / G_i$, 
then $G_{i,j}$ is a $V$-generic filter on $\p_j$, 
$G_{i,j} \cap \p_i = G_i$, and 
$V[G_i][G_{i,j}] = V[G_{i,j}]$. 
In particular, if $j \in A$, then in $V[G_i][G_{i,j}] = V[G_{i,j}]$, 
$\dot T_j^{G_{i,j}}$ is a tree with 
underlying set $\omega_1$ which has no branches 
of length $\omega_1$. 
In this situation, we will write 
$\dot T_{i,j}$ for a $\p_j / G_i$-name in $V[G_i]$ for this tree.

\begin{proposition}
Let $i < \lambda$, and let $G_i$ be a $V$-generic filter on $\p_i$. 
Then in $V[G_i]$, there exists a special $(A \setminus i)$-iteration 
$\langle \p_j' : j \le \lambda \rangle$, with sequence of 
names $\langle \dot T_j' : j \in A \setminus i \rangle$, satisfying 
the following properties:
\begin{enumerate}

\item for each $j \le \lambda$, the map 
$p \mapsto p \restriction (A \setminus i)$ is a dense embedding 
of $\p_j / G_i$ into $\p_j'$;

\item for each $j \in A \setminus i$, $\dot T_j'$ is the canonical translation 
of the name $\dot T_{i,j}$ to a $\p_j'$-name using the dense embedding from (1).
\end{enumerate}
\end{proposition}

\begin{proof}
First consider $j \le i$. 
Then $\p_j'$ is the trivial poset $\{ \emptyset \}$, and 
$\p_j / G_i$ is equal to $G_i \cap \p_j$. 
Using the fact that $G_i \cap \p_j$ is a filter to show preservation of 
incompatibility, 
easily the map $p \mapsto p \restriction (A \setminus i) = \emptyset$ 
is a dense embedding of $\p_j / G_i = G_j \cap \p_j$ into 
$\p_j' = \{ \emptyset \}$.

Suppose that $i \le \beta \le \lambda$ is limit ordinal, and 
for all $\gamma < \beta$, the map 
$p \mapsto p \restriction (A \setminus i)$ is a dense embedding of 
$\p_\gamma / G_i$ into $\p_\gamma'$. 
Using the fact that $\p_\beta / G_i = \bigcup \{ \p_\gamma / G_i : 
\gamma < \beta \}$ and $\p_\beta' = \bigcup \{ \p_\gamma' : \gamma < \beta \}$, 
it is easy to check that the same is true of $\p_\beta / G_i$ 
and $\p_\beta'$.

Finally, assume that $i < j < \lambda$, 
and the map $p \mapsto p \restriction (A \setminus i)$ 
is a dense embedding of $\p_j / G_i$ into $\p_j'$. 
We will prove that the same is true for $\p_{j+1} / G_i$ and $\p_{j+1}'$. 
First assume that $j \notin A$. 
Then $\p_{j+1} = \p_j$, so $\p_{j+1} / G_i = \p_j / G_i$. 
Also $\p_{j+1}' = \p_j'$. 
So we are done by the inductive hypothesis.

Suppose that $j \in A$. 
In $V[G_i]$, $\dot T_{i,j}$ is a $\p_j / G_i$-name for a tree with underlying 
set $\omega_1$ which has no branches of length $\omega_1$. 
Since $p \mapsto p \restriction (A \setminus i)$ 
is a dense embedding of $\p_j / G_i$ into $\p_j'$, 
the name $\dot T_{i,j}$ translates under this dense embedding 
to a $\p_j'$-name $\dot T_j'$ 
for the same tree. 

Assume that $q \in \p_{j+1} / G_i$, and we will show that 
$q \restriction (A \setminus i)$ is in $\p_{j+1}'$. 
This follows from the inductive hypothesis if $j \notin \dom(q)$, so 
suppose that $j \in \dom(q)$. 
Since $q \restriction j \in \p_j / G_i$, 
by the inductive hypothesis $(q \restriction j) \restriction (A \setminus i) = 
(q \restriction (A \setminus i)) \restriction j$ is in $\p_j'$. 
As $q \in \p_{j+1}$, $q \restriction j \Vdash^V_{\p_j} q(j) \in P(\dot T_j)$. 
By the choice of $\dot T_{i,j}$, 
$q \restriction j \Vdash^{V[G_i]}_{\p_j / G_i} q(j) \in P(\dot T_{i,j})$. 
Hence $(q \restriction (A \setminus i)) \restriction j \Vdash^{V[G_i]}_{\p_j'} 
q(j) \in P(\dot T_j')$. 
So $q \restriction (A \setminus i) \in \p_{j+1}'$.

Let $p \in \p_{j+1}'$. 
We will find a condition 
$u$ in $\p_{j+1} / G_i$ 
such that $u \restriction (A \setminus i) \le p$. 
If $j \notin \dom(p)$, then $p \in \p_j'$. 
By the inductive hypothesis, there is $q \in \p_j / G_i$ such that 
$q \restriction (A \setminus i) \le p$. 
Since $\p_j \subseteq \p_{j+1}$ and $q \restriction i \in G_i$, 
$q \in \p_{j+1} / G_i$, and we are done.

Suppose that $j \in \dom(p)$. 
Let $x := p(j)$. 
By the inductive hypothesis, fix $q \in \p_j / G_i$ such that 
$q \restriction (A \setminus i) \le p \restriction i$. 
The proof would be finished if $q \cup \{ (j,x) \}$ was in $\p_{j+1} / G_i$, 
that is, if $q \cup \{ (j,x) \}$ was in $\p_{j+1}$. 
Unfortunately, we do not know that $q$ forces over $V$ that 
$x$ is in $P(\dot T_j)$, so it could be the case that 
$q \cup \{ (j,x) \}$ is not in $\p_{j+1}$. 
So we have to work a bit harder.

Since $q \restriction (A \setminus i) \le p \restriction i$, 
it follows that 
$$
q \restriction (A \setminus i) \Vdash^{V[G_i]}_{\p_j'} x \in P(\dot T_j').
$$
We claim that 
$$
q \Vdash^{V[G_i]}_{\p_j / G_i} x \in P(\dot T_{i,j}).
$$
Let $G_{i,j}$ be a $V[G_i]$-generic filter on $\p_j / G_i$ with 
$q \in G_{i,j}$. 
Let $T := \dot T_{i,j}^{G_{i,j}}$. 
We will prove that $x \in P(T)$. 
The image of $G_{i,j}$ under the dense embedding 
$s \mapsto s \restriction (A \setminus i)$ generates a 
$V[G_i]$-generic filter $G_j'$ on $\p_j'$. 
Since $q \in G_{i,j}$, $q \restriction (A \setminus i) \in G_j'$. 
And since $q \restriction (A \setminus i) \le p \restriction j$, 
$p \restriction j \in G_j'$. 
As $p \in \p_{j+1}'$, by definition we have that 
$$
p \restriction j \Vdash^{V[G_i]}_{\p_j'} p(j) = x \in P(\dot T_j').
$$
Since $p \in G_j'$, it follows that $x \in P((\dot T_j')^{G_j'})$. 
But by the choice of the name $\dot T_j'$, 
$(\dot T_j')^{G_j'} = (\dot T_{i,j})^{G_{i,j}} = T$. 
So $x \in P(T)$, proving the claim.

Fix a condition $s \in G_i$ such that 
$$
s \Vdash^V_{\p_i} q \Vdash^{V[\dot G_i]}_{\p_j / \dot G_i} x \in 
P(\dot T_{i,j}).
$$
Since $q \restriction i \in G_i$, without loss of generality we may 
assume that $s \le q \restriction i$. 
It follows that $t := s \cup (q \restriction [i,\lambda))$ is in $\p_j$, 
$t \le q$, and since $t \restriction i = s \in G_i$, 
$t \in \p_j / G_i$. 
Now the choice of $s$ implies that 
$$
t \Vdash^V_{\p_j} p(j) \in \dot T_j,
$$
as can be easily checked. 
So $u := t \cup \{ (j,x) \}$ is in $\p_{j+1} / G_i$, and 
$u \restriction (A \setminus i) \le p$.

It is obvious that the map $p \mapsto p \restriction (A \setminus i)$ 
preserves order. 
For the preservation of incompatibility, suppose that $p$ and $q$ are 
in $\p_{j+1} / G_i$, and $r \le p \restriction (A \setminus i), 
q \restriction (A \setminus i)$ in $\p_{j+1}'$. 
By what we just proved, there is $r_0 \in \p_{j+1} / G_i$ 
such that $r_0 \restriction (A \setminus i) \le r$ in $\p_{j+1}'$. 
Since $G_i$ is a filter, without loss of generality we may assume that 
$r_0 \restriction i \le p \restriction i, q \restriction i$. 
Then easily $r_0 \le p, q$ in $\p_{j+1} / G_i$.
\end{proof}

\begin{corollary}
Let $\langle \p_i : i \le \lambda \rangle$ be a special $A$-iteration. 
Fix $i < \lambda$, and let $G_i$ be a $V$-generic filter on $\p_i$. 
Then in $V[G_i]$, for all $j \le \lambda$, 
$\p_j / G_i$ is $\omega_1$-c.c.\ and has the 
$\omega_1$-approximation property.
\end{corollary}

\begin{proof}
By Proposition 6.9, 
$\p_j / G_i$ is forcing equivalent to a special $(A \setminus i)$-iteration. 
So we are done by Corollary 6.6.
\end{proof}

\section{Factoring a special iteration over an elementary substructure}

In this section we will prove the following result.

\begin{thm}
Let $\vec P = \langle \p_i : i \le \lambda \rangle$ be a special iteration, 
with sequence of names $\vec T = \langle \dot T_i : i < \lambda \rangle$. 
Let $\theta \ge \omega_2$ be a regular cardinal such that 
$\vec P$ and $\vec T$ are in $H(\theta)$. 
Let $N$ be an elementary substructure of $H(\theta)$ such that 
$N$ has size $\omega_1$, $\omega_1 \subseteq N$, and  
$\vec P$ and $\vec T$ are in $N$.

Then:
\begin{enumerate}
\item $N \cap \p_\lambda$ is a regular suborder of $\p_\lambda$;
\item $N \cap \p_\lambda$ forces that 
$\p_\lambda / \dot G_{N \cap \p_\lambda}$ 
is forcing equivalent to a special 
$(\lambda \setminus N)$-iteration;
\item $N \cap \p_\lambda$ forces that 
$\p_\lambda / \dot G_{N \cap \p_\lambda}$ 
is $\omega_1$-c.c.\ and has the $\omega_1$-approximation property.
\end{enumerate}
\end{thm}

Note that (3) follows from (2) by Corollary 6.6.

The proof of Theorem 7.1(1) is given in Lemma 7.5(2). 
The rest of the section after that is devoted to proving Theorem 7.1(2). 
The proof of Theorem 7.1(2) is tedious; we suggest that the 
reader skip it on a first reading.

Fix for the remainder of the section $\vec P$, $\vec T$, $\theta$, and $N$ 
as described in Theorem 7.1.

\begin{lemma}
Let $i < j \le \lambda$, where $j \in N$.
\begin{enumerate}
\item the map $p \mapsto p \restriction i$ is a projection mapping from 
$\p_j \cap N$ into $\p_i \cap N$;
\item $p \restriction i = p$ for all $p \in \p_i \cap N$, and 
$q \le q \restriction i$ for all $q \in \p_j \cap N$;
\item $\p_i \cap N$ is a regular suborder of $\p_j \cap N$.
\end{enumerate}
\end{lemma}

\begin{proof}
(2) follows from Lemma 6.3(4), and (3) 
follows from (1), (2), and Lemma 1.5. 
It remains to prove (1). 
Let $i_N := \min((N \cap (\lambda+1)) \setminus i)$. 
By Lemma 6.3(4), the map $p \mapsto p \restriction i_N$ is a projection 
mapping of $\p_j$ into $\p_{i_N}$. 
Using the elementarity of $N$, it easily follows  
that the same map restricted to $\p_j \cap N$ is a projection 
mapping of $\p_j \cap N$ into $\p_{i_N} \cap N$. 
But if $q \in \p_{i_N}$, then by elementarity, 
$\dom(q) \subseteq i_N \cap N \subseteq i$, 
and hence $q \in \p_i$. 
It follows that $\p_{i_N} \cap N = \p_i \cap N$, and 
$p \restriction i = p \restriction i_N$ for all $p \in \p_j \cap N$. 
So this map is a projection mapping of $\p_j \cap N$ into $\p_i \cap N$.
\end{proof}

\begin{definition}
For each $i \le \lambda$, let $E_i$ denote the set of $p \in \p_i$ such that 
$p \restriction N$, the restriction of the function $p$ to $N \cap i$, is in $\p_i$.
\end{definition}

\begin{lemma}
Let $i < j \le \lambda$. 
Then:
\begin{enumerate}
\item $E_i \subseteq E_j$;
\item if $p \in E_j$, then $p \restriction i \in E_i$.
\end{enumerate}
\end{lemma}

\begin{proof}
Let $p \in E_i$. 
Then $p \in \p_i$ and $p \restriction N \in \p_i$. 
Since $\p_i \subseteq \p_j$ by Lemma 6.3(1), 
it follows that $p \in \p_j$ and $p \restriction N \in \p_j$. 
So $p \in E_j$, which proves (1). 
For (2), let $p \in E_j$. 
Then $p \in \p_j$ and $p \restriction N \in \p_j$. 
By Lemma 6.3(2), $p \restriction i \in \p_i$ and 
$(p \restriction i) \restriction N = (p \restriction N) \restriction i 
\in \p_i$. 
So $p \restriction i \in E_i$.
\end{proof}

We now prove Theorem 7.1(1). 
We also prove that $E_i$ is dense in $\p_i$, since that fact 
follows by the same argument.

\begin{lemma}
Let $i \le \lambda$. 
Then:
\begin{enumerate}
\item $E_i$ is dense in $\p_i$; 
\item $N \cap \p_i$ is a regular suborder of $\p_i$.
\end{enumerate}
\end{lemma}

\begin{proof}
Suppose that $p$ and $q$ are in $N \cap \p_i$, and are compatible in $\p_i$. 
We will show that $p$ and $q$ are compatible in $\p_i \cap N$. 
Let $i_N := \min((N \cap (\lambda+1)) \setminus i)$. 
Since $\p_i$ is a regular suborder of $\p_{i_N}$ 
by Lemma 6.3(5), $p$ and $q$ are 
compatible in $\p_{i_N}$. 
By elementarity, there is $r \in N \cap \p_{i_N}$ such that 
$r \le p, q$. 
Then $\dom(r) \subseteq N \cap i_N \subseteq i$, so $r \in \p_i$. 
Hence $r \in \p_i \cap N$ and $r \le p, q$.

Let $A$ be a maximal antichain of $N \cap \p_i$, and we will show 
that $A$ is predense in $\p_i$. 
Let $p \in \p_i$. 
We will find $r \in \p_i$ and $s \in A$ such that $r \le s, p$. 
Moreover, we will have that 
$r \restriction N \in \p_i$, and hence that $r \in E_i$. 
This proves both that $A$ is predense in $\p_i$, and that $E_i$ 
is dense in $\p_i$.

Since $\omega_1 \subseteq N$, $\textrm{Fn}(\omega_1,\omega)$ 
is a subset of $N$. 
It follow that $p \restriction N$ is equal to $p \cap N$, which is in $N$ 
since $p \cap N$ is finite. 
However, we do not know that $p \cap N$ is in $\p_i$. 
Since $p \in \p_i \subseteq \p_{i_N}$, 
by the elementarity of $N$ we can fix $p'$ in $\p_{i_N} \cap N$ such that 
$p \restriction N \subseteq p'$.

Since $A$ is a maximal antichain in $\p_i \cap N$, $A$ is also 
a maximal antichain of $\p_{i_N} \cap N$. 
So we can fix $s \in A$ 
and $q \in \p_{i_N} \cap N$ such that $q \le p', s$. 
Then by elementarity, $\dom(q) \subseteq i_N \cap N \subseteq i$, 
so $q \in \p_i \cap N$. 
As $\dom(q) \subseteq N$ and 
$\dom(p \restriction N) \subseteq \dom(p') \subseteq \dom(q)$, we 
have that $\dom(p) \cap \dom(q) = \dom(p \restriction N)$. 
And since $p \restriction N \subseteq p'$ and $q \le p'$, 
it follows that for all $i \in \dom(p \restriction N)$,  $p(i) \subseteq q(i)$. 
By Lemma 6.4, $r := p + q$ exists and $r \le p, q$. 
Hence $r \le p, s$. 
Moreover, by the definition of $p + q$, 
$r \restriction N = q \in \p_i$, so $r \in E_i$.
\end{proof}

\begin{notation}
Fix $i \le \lambda$, and let $G_{i,N}$ be a $V$-generic filter on $\p_i \cap N$. 
Let $E_i^N$ denote the set $E_i \cap (\p_i / G_{i,N})$.
\end{notation}
 
Since $E_i$ is dense in $\p_i$, $E_i^N$ is dense in $\p_i / G_{i,N}$ by 
Lemma 1.3.

\begin{lemma}
Fix $i \le \lambda$, and let $G_{i,N}$ be a $V$-generic 
filter on $\p_i \cap N$. 
Then for all $p \in E_i$, $p \in E_i^N$ iff 
$p \restriction N \in G_{i,N}$.
\end{lemma}

\begin{proof}
Using Lemma 6.4, it is easy to check that the map $p \mapsto p \restriction N$ 
is a projection mapping from $E_i$ into $E_i \cap N = \p_i \cap N$, 
which satisfies that $p \restriction N = p$ for all $p \in E_i \cap N$, 
and $q \le q \restriction N$ for all $q \in E_i$. 
By Lemma 1.5, it follows that 
the set of $p \in E_i$ such that $p$ is compatible in $E_i$ 
with every condition in $G_{i,N}$ is equal to the set of $p \in E_i$ 
such that $p \restriction N \in G_{i,N}$. 
But a condition in $E_i$ is compatible in $E_i$ with every condition in 
$G_{i,N}$ iff it is in 
$E_i \cap (\p_i / G_{i,N}) = E_i^N$.
\end{proof}

Given a $V$-generic filter $G_N$ on $\p_\lambda \cap N$,  
for each $i \le \lambda$, let $G_{i,N} := G_N \cap \p_i$. 
Since $\p_i \cap N$ is a regular suborder of $\p_\lambda \cap N$, 
$G_{i,N}$ is a $V$-generic filter on $\p_i \cap N$. 
Also, for all $i < j \le \lambda$, we have that 
$G_{i,N} \subseteq G_{j,N}$, and for all $s \in G_{j,N}$, 
$s \restriction i \in G_{i,N}$.

\begin{lemma}
Let $i < j \le \lambda$. Then:
\begin{enumerate}
\item $E_i^N \subseteq E_j^N$;
\item if $p \in E_j^N$, then $p \restriction i \in E_i^N$.
\end{enumerate}
\end{lemma}

\begin{proof}
(1) Let $p \in E_i^N$. 
Then $p \in E_i$, and by Lemma 7.7, $p \restriction N \in G_{i,N}$. 
But $E_i \subseteq E_j$ by Lemma 7.4(1), and $G_{i,N} \subseteq G_{j,N}$. 
So $p \in E_j$ and $p \restriction N \in G_{j,N}$. 
By Lemma 7.7, $p \in E_j^N$.

(2) Let $p \in E_j^N$. 
Then by Lemma 7.7, $p \in E_j$ and $p \restriction N \in G_{j,N}$. 
By Lemma 7.4(2), $p \restriction i \in E_i$, and by the comments preceding this 
lemma, $(p \restriction i) \restriction N = (p \restriction N) \restriction i$ 
is in $G_{i,N}$. 
By Lemma 7.7, $p \restriction i \in E_i^N$.
\end{proof}

\begin{notation}
Let $i \le \lambda$. 
In $V[G_{i,N}]$, define $\p_i^N$ as the suborder of $S(i)$ consisting of 
functions of the form $p \restriction (i \setminus N)$, where $p \in E_i^N$.
\end{notation}

So $\p_i^N$ consists of functions in $E_i^N$, with their fragment belonging 
to $N$ removed.

\begin{lemma}
Let $i < j \le \lambda$. Then:
\begin{enumerate}
\item $\p_i^N \subseteq \p_j^N$;
\item for all $p \in \p_j^N$, $p \restriction i \in \p_i^N$. 
\end{enumerate}
\end{lemma}

\begin{proof}
(1) Let $p \in \p_i^N$. 
Then there is $p_0 \in E_i^N$ such that 
$p = p_0 \restriction (i \setminus N)$. 
Since $E_i^N \subseteq E_j^N$ by Lemma 7.8(1), 
$p_0 \in E_j^N$. 
Easily $p = p_0 \restriction (j \setminus N)$, so $p \in \p_j^N$.

(2) Let $p \in \p_j^N$. 
Then there is $p_0 \in E_j^N$ such that 
$p = p_0 \restriction (j \setminus N)$. 
By Lemma 7.8(2), $p_0 \restriction i \in E_i^N$, and 
$p \restriction i = (p_0 \restriction (j \setminus N)) \restriction i 
= (p_0 \restriction i) \restriction (i \setminus N)$. 
So $p \restriction i \in \p_i^N$.
\end{proof}

\begin{definition}
Define $\pi_i : E_i^N \to \p_i^N$ in $V[G_{i,N}]$ 
by letting 
$$
\pi_i(p) = p \restriction (i \setminus N),
$$
for all $p \in E_i^N$.
\end{definition}

\begin{lemma}
The function $\pi_i$ is a dense embedding of $E_i^N$ onto $\p_i^N$. 
Hence $\p_i^N$ is forcing equivalent to $\p_i / G_{i,N}$.
\end{lemma}

\begin{proof}
The second statement follows from the first together with the fact 
that $E_i^N$ is dense in $\p_i / G_{i,N}$. 
The function $\pi_i$ is surjective by the definition of $\p_i^N$, 
and clearly satisfies that 
$q \le p$ in $S(i)$ implies that $\pi_i(q) \le \pi_i(p)$ in $S(i)$. 
To see that $\pi_i$ preserves incompatibility, suppose that $p$ and $q$ 
are in $E_i^N$, and $r \le \pi_i(p), \pi_i(q)$ in $\p_i^N$. 
We will show that $p$ and $q$ are compatible in $E_i^N$. 
Fix $r_0 \in E_i^N$ such that 
$r = r_0 \restriction (i \setminus N)$.

Since $r_0$, $p$, and $q$ are in $E_i^N$, it follows that 
$r_0 \restriction N$, $p \restriction N$, and $q \restriction N$ 
are in $G_{i,N}$ by Lemma 7.7. 
As $G_{i,N}$ is a filter, we can fix $s \in G_{i,N}$ with 
$s \le r_0 \restriction N, p \restriction N, q \restriction N$. 
Then $\dom(s) \cap \dom(r_0) = \dom(s)$, and for all 
$\gamma \in \dom(s)$, $r_0(\gamma) \subseteq s(\gamma)$. 
By Lemma 6.4, $r_0 + s$ is a condition in $\p_i$ which is below $r_0$ and $s$. 
Moreover, $(r_0 + s) \restriction N = s \in G_{i,N}$, 
so $r_0 + s \in E_i^N$. 
Now $(r_0 + s) \restriction N = s \le p \restriction N, q \restriction N$, 
and $(r_0 + s) \restriction (i \setminus N) = 
r \le \pi_i(p) = p \restriction (i \setminus N), 
\pi_i(q) = q \restriction (i \setminus N)$. 
So $r_0 + s \le p, q$.
\end{proof}

Recall that for $i \le \lambda$, 
$\p_i \cap N$ is a regular suborder of both $\p_i$ and $\p_\lambda \cap N$, 
by Lemma 7.2 and Lemma 7.5(2). 
So if $G_i$ is a $V$-generic filter on $\p_i$, then 
$G_i \cap N$ is a $V$-generic filter on $\p_i \cap N$. 
Hence we can form the forcing poset $(\p_\lambda \cap N) / (G_i \cap N)$.

\begin{lemma}
Let $i \in N \cap (\lambda+1)$. Then:
\begin{enumerate}
\item $\p_i$ forces that 
$(\p_\lambda \cap N) / (\dot G_i \cap N)$ is equal to 
$(\p_\lambda / \dot G_i) \cap N[\dot G_i]$;
\item $\p_i$ forces that $(\p_\lambda \cap N) / (\dot G_i \cap N)$ is 
$\omega_1$-c.c.\ and has the $\omega_1$-approximation property.
\end{enumerate}
\end{lemma}

\begin{proof}
(1) Let $G_i$ be a $V$-generic filter on $\p_i$. 
We will show that 
$$
(\p_\lambda \cap N) / (G_i \cap N) = 
(\p_\lambda / G_i) \cap N[G_i].
$$
Let $p \in (\p_\lambda \cap N) / (G_i \cap N)$. 
Then by the definition of the quotient forcing, 
$p \in \p_\lambda \cap N$ and 
$p \restriction i \in G_i \cap N$. 
In particular, $p \restriction i \in G_i$. 
So $p \in \p_\lambda / G_i$. 
Also, since $p \in N$ and $N \subseteq N[G_i]$, it follows that 
$p \in N[G_i]$. 
Hence $p \in (\p / G_i) \cap N[G_i]$.

Conversely, let $p \in (\p_\lambda / G_i) \cap N[G_i]$. 
By the definition of the quotient forcing, 
$p \in \p_\lambda$ and $p \restriction i \in G_i$. 
Since $\p_i$ is $\omega_1$-c.c., $N[G_i] \cap V = N$ by standard arguments. 
Therefore $p \in N[G_i] \cap \p_\lambda = N \cap \p_\lambda$. 
So $p \in \p_\lambda \cap N$. 
Since $p$ and $i$ are in $N$, $p \restriction i \in N$. 
So $p \restriction i \in G_i \cap N$. 
Hence $p \in (\p_\lambda \cap N) / (G_i \cap N)$.

(2) By Proposition 6.9, there exists in $V[G_i]$ a special 
$(\lambda \setminus i)$-iteration $\p_\lambda'$ such that the map 
$f(p) = p \restriction (\lambda \setminus i)$ is a dense embedding 
of $\p_\lambda / G_i$ into $\p_\lambda'$. 
By the elementarity of $N[G_i]$, we can choose $\p_\lambda'$ 
to be in $N[G_i]$. 
Again by elementarity, it is easy to check that $f \restriction N[G_i]$ 
is a dense embedding of $(\p_\lambda / G_i) \cap N[G_i]$ 
into $\p_\lambda' \cap N[G_i]$. 
Hence $(\p_\lambda / G_i) \cap N[G_i]$ is forcing equivalent 
to $\p_\lambda' \cap N[G_i]$. 
By (1), it follows that $(\p_\lambda \cap N) / (G_i \cap N)$ 
is forcing equivalent to $\p_\lambda' \cap N[G_i]$.

By Lemma 7.5(2) applied in $V[G_i]$ to $\p_\lambda'$ and $N[G_i]$, 
$\p_\lambda' \cap N[G_i]$ is a regular suborder of 
$\p_\lambda'$. 
But $\p_\lambda'$ is a special $(\lambda \setminus i)$-iteration, and 
hence is $\omega_1$-c.c.\ and has the $\omega_1$-approximation property. 
Since $\p_\lambda' \cap N[G_i]$ is a regular suborder of 
$\p_\lambda'$, it is also $\omega_1$-c.c., and by Lemma 1.8, it has 
the $\omega_1$-approximation property. 
Hence $(\p_\lambda \cap N) / (G_i \cap N)$ is $\omega_1$-c.c.\ and 
has the $\omega_1$-approximation property.
\end{proof}

\begin{lemma}
Let $i \le \lambda$. 
Then $\p_i$ forces that 
$(\p_\lambda \cap N) / (\dot G_i \cap N)$ has the $\omega_1$-approximation property.
\end{lemma}

\begin{proof}
Let $G_i$ be a $V$-generic filter on $\p_i$, and let 
$H$ be a $V[G_i]$-generic filter on $(\p_\lambda \cap N) / (G_i \cap N)$. 
We will show that the pair $(V[G_i],V[G_i][H])$ has the 
$\omega_1$-approximation property. 

Let $j := \min((N \cap (\lambda+1)) \setminus i)$. 
Let $G_{i,j}$ be a $V[G_i][H]$-generic filter on $\p_j / G_i$. 
By the product lemma and Lemma 1.1, 
$$
V[G_i][H][G_{i,j}] = 
V[G_i][G_{i,j}][H] = V[G_{i,j}][H],
$$
and $G_{i,j}$ is a $V$-generic filter on $\p_j$ such that 
$G_{i,j} \cap \p_i = G_i$. 

We claim that $G_i \cap N = G_{i,j} \cap N$. 
Since $G_i \subseteq G_{i,j}$, the forward inclusion is immediate. 
Conversely, let $p \in G_{i,j} \cap N$, and we will show 
that $p \in G_i$. 
Then $p \in \p_j \cap N$. 
By elementarity, $\dom(p) \subseteq N \cap j \subseteq i$, so 
$p \in G_{i,j} \cap \p_i = G_i$. 
It follows from the claim that 
$(\p_\lambda \cap N) / (G_i \cap N) = 
(\p_\lambda \cap N) / (G_{i,j} \cap N)$.

Since $j \in N$, by Lemma 7.13 it follows that 
$(\p_\lambda \cap N) / (G_{i,j} \cap N)$ has the 
$\omega_1$-c.c.\ and the $\omega_1$-approximation property in $V[G_{i,j}]$. 
Therefore the pair 
$$
(V[G_{i,j}],V[G_{i,j}][H])
$$
has the 
$\omega_1$-covering property and 
the $\omega_1$-approximation property. 
That is, the pair 
$$
(V[G_i][G_{i,j}],V[G_i][G_{i,j}][H])
$$ 
has the $\omega_1$-covering property and the $\omega_1$-approximation property. 
By Corollary 6.10, the forcing poset $\p_j / G_i$ is $\omega_1$-c.c.\ 
and has the $\omega_1$-approximation property in $V[G_i]$. 
So the pair 
$$
(V[G_i],V[G_i][G_{i,j}])
$$
has the $\omega_1$-covering property and the 
$\omega_1$-approximation property. 
By Lemma 1.7, it follows that the pair 
$$
(V[G_i],V[G_i][G_{i,j}][H])
$$
has the $\omega_1$-approximation property. 
Since $V[G_i] \subseteq V[G_i][H] \subseteq V[G_i][G_{i,j}][H]$, 
by Lemma 1.8 the pair 
$$
(V[G_i],V[G_i][H])
$$
has the $\omega_1$-approximation property.
\end{proof}

For the remainder of the section, fix a $V$-generic filter 
$G_N$ on $\p_\lambda \cap N$. 
Our goal is to prove that in $V[G_N]$, $\p_\lambda / G_N$ 
is forcing equivalent to a special $(\lambda \setminus N)$-iteration. 
For $i \le \lambda$, let $G_{i,N} := G_N \cap \p_i$. 
Recall that $E_i^N$ is dense in $\p_i / G_{i,N}$, and 
$\pi_i : E_i^N \to \p_i^N$ is a dense embedding, 
where $\pi_i(p) := p \restriction (i \setminus N)$ for 
all $p \in E_i^N$. 
So it suffices to show that in $V[G_N]$, 
$\p_i^N$ is forcing equivalent to a 
special $(i \setminus N)$-iteration, for all $i \le \lambda$.

\begin{notation}
In $V[G_{i,N}]$, 
let $\dot G_i$ be a $\p_i^N$-name for the $V[G_{i,N}]$-generic filter 
on $\p_i / G_{i,N}$ generated by $\pi_i^{-1}(\dot G_{\p_i^N})$. 
Also in $V[G_{i,N}]$, let $\dot T_i^N$ be a $\p_i^N$-name for 
$(\dot T_i)^{\dot G_i}$.
\end{notation}

\begin{lemma}
The forcing poset $\p_i^N$ forces over $V[G_N]$ that 
$\dot T_i^N$ is a tree with underlying set $\omega_1$ 
which has no branches of length $\omega_1$.
\end{lemma}

\begin{proof}
Let $G_i^N$ be a $V[G_N]$-generic filter on $\p_i^N$. 
We will show that in $V[G_N][G_i^N]$, $(\dot T_i^N)^{G_i^N}$ 
is a tree with underlying set $\omega_1$ which has no branches 
of length $\omega_1$. 
Let $G_i$ denote the filter on $\p_i / G_{i,N}$ generated by 
$\pi_i^{-1}(G_i^N)$. 
Since $\pi_i$ is a dense embedding and $E_i^N$ is dense in $\p_i / G_{i,N}$, 
$G_i$ is a $V[G_N]$-generic filter on $\p_i / G_{i,N}$. 
Moreover, 
$$
V[G_N][G_i^N] = V[G_N][G_i].
$$

We claim that $G_i \cap N = G_{i,N}$ and 
$V[G_{i,N}][G_i] = V[G_i]$. 
Since $G_i$ is a $V[G_N]$-generic filter on $\p_i / G_{i,N}$, it is also 
a $V[G_{i,N}]$-generic filter on $\p_i / G_{i,N}$. 
By Lemma 1.1, $G_i$ is a $V$-generic filter on $\p_i$, 
$G_i \cap (\p_i \cap N) = G_{i,N}$, and 
$V[G_{i,N}][G_i] = V[G_i]$. 
But $G_i \cap (\p_i \cap N) = G_i \cap N$, proving the claim.

Note that both of the forcing posets $\p_i / G_{i,N}$ and 
$(\p_\lambda \cap N) / G_{i,N}$ are in $V[G_{i,N}]$. 
By Lemma 1.1,
$$
V[G_N] = V[G_{i,N}][G_N],
$$
and $G_N$ is a $V[G_{i,N}]$-generic filter on 
$(\p_\lambda \cap N) / G_{i,N}$. 
So $G_i$ is a $V[G_{i,N}][G_N]$-generic filter on $\p_i / G_{i,N}$. 
By the product lemma,
$$
V[G_N][G_i] = 
V[G_{i,N}][G_N][G_i] = 
V[G_{i,N}][G_i][G_N],
$$
and $G_N$ is a $V[G_{i,N}][G_i]$-generic filter 
on $(\p_\lambda \cap N) / G_{i,N}$.

By the above claim, $G_{i,N} = G_i \cap N$ 
and $V[G_{i,N}][G_i] = V[G_i]$. 
So 
$$
V[G_N][G_i] = V[G_{i,N}][G_i][G_N] = V[G_i][G_N],
$$
and $G_N$ is a $V[G_i]$-generic filter on $(\p_\lambda \cap N) / (G_i \cap N)$. 
By Lemma 7.14, 
it follows that in $V[G_i]$, $(\p_\lambda \cap N) / (G_i \cap N)$ 
has the $\omega_1$-approximation property. 
Hence the pair $(V[G_i],V[G_i][G_N])$ has the $\omega_1$-approximation 
property.

Recall that $\dot T_i$ is a $\p_i$-name for a tree with underlying set $\omega_1$ 
which has no branches of length $\omega_1$. 
Let $T_i := \dot T_i^{G_i}$. 
Then in $V[G_i]$, $T_i$ is a tree with underlying set $\omega_1$ 
which has no branches of length $\omega_1$. 
By upwards absoluteness, 
in $V[G_N][G_i^N] = V[G_N][G_i] = V[G_i][G_N]$, 
$T_i$ is a tree with underlying set $\omega_1$. 
Since the pair $(V[G_i],V[G_i][G_N])$ has the $\omega_1$-approximation 
property, $T_i$ has no branches of length $\omega_1$ in $V[G_i][G_N]$. 
By Notation 7.15, $(\dot T_i^N)^{G_i^N}$ is equal to 
$\dot T_i^{G_i} = T_i$, and this is a tree with underlying set $\omega_1$ with 
no branches of length $\omega_1$ in $V[G_i][G_N] = V[G_N][G_i] = V[G_N][G_i^N]$. 
\end{proof}

Let us return to proving that in $V[G_N]$, 
$\p_i^N$ is forcing equivalent to a special $(i \setminus N)$-iteration, 
for all $i \le \lambda$. 
In the model $V[G_N]$, consider the sequence 
$$
\langle \p_i' : i \le \lambda \rangle
$$
which is the special $(\lambda \setminus N)$-iteration defined from the 
sequence $\langle \dot T_i^N : i \in \lambda \setminus N \rangle$. 
In other words, the sequence is defined inductively using (3)--(6) of Definition 6.2. 
Of course, this only makes sense if for each $i \in \lambda \setminus N$, 
the name $\dot T_i^N$ from Notation 7.15 
is a $\p_i'$-name for a tree with underlying set $\omega_1$ 
which has no branches of length $\omega_1$.

We will prove inductively that for each $i \le \lambda$, 
$\p_i^N$ is a dense subset of $\p_i'$. 
Then for all $i \in \lambda \setminus N$, 
$\dot T_i^N$ literally is a $\p_i'$-name. 
And since $\p_i^N$ forces over $V[G_N]$ that 
$\dot T_i^N$ is a tree with underlying set $\omega_1$ which has  
no branches of length $\omega_1$ by Lemma 7.16, so does $\p_i'$. 
So in the end, we have that 
$\p_\lambda^N$ is forcing equivalent to $\p_\lambda'$, which is a 
special $(\lambda \setminus N)$-iteration, completing the proof of Theorem 7.1.

It remains to prove the following lemma.

\begin{lemma}
For all $i \le \lambda$, $\p_i^N$ is a dense subset of $\p_i'$.
\end{lemma}

\begin{proof}
This is trivial for $i = 0$ by Definition 6.2(3), 
and it follows easily from the inductive hypothesis 
when $i$ is a limit ordinal by Definition 6.2(6). 

Let $i < \lambda$, and suppose that $\p_i^N$ is a dense subset of $\p_i'$. 
We will prove that $\p_{i+1}^N$ is a dense subset of $\p_{i+1}'$. 
First, assume the easier case that $i \in N$. 
Then $\p_{i+1}' = \p_i'$ by Definition 6.2(5). 
So it suffices to show that $\p_{i+1}^N = \p_{i}^N$. 
But $\p_i^N \subseteq \p_{i+1}^N$ by Lemma 7.10(1). 
Conversely, let $p \in \p_{i+1}^N$, and we will show that $p \in \p_i^N$. 
Then for some $p_0 \in E_{i+1}^N$, 
$p = p_0 \restriction ((i+1) \setminus N)$. 
But $i \in N$, so $(i+1) \setminus N = i \setminus N$. 
Hence $p = p_0 \restriction ((i+1) \setminus N) = 
p_0 \restriction (i \setminus N) = 
(p_0 \restriction i) \restriction (i \setminus N)$. 
But $p_0 \restriction i \in E_i^N$ by Lemma 7.8(2), 
so $p \in \p_i^N$.

Secondly, assume that $i \notin N$. 
We will show that $\p_{i+1}^N$ is dense in $\p_{i+1}'$. 
Let $p \in \p_{i+1}'$, and we will find a condition in $\p_{i+1}^N$ 
which is below $p$. 
If $i \notin \dom(p)$, then $p = p \restriction i \in \p_i'$. 
By the inductive hypothesis, there is $q \le p$ in $\p_i^N$. 
Then $q \in \p_{i+1}^N$ by Lemma 7.10(1), 
and $q \le p$, so we are done.

Suppose that $i \in \dom(p)$. 
Then by the definition of $\p_{i+1}'$, 
$p \restriction i \in \p_i'$, and 
$$
p \restriction i \Vdash^{V[G_N]}_{\p_i'} p(i) \in P(\dot T_i^N).
$$
Let $x = p(i)$. 
By the inductive hypothesis, we can fix $q \le p \restriction i$ in $\p_i^N$. 
Then 
$$
q \Vdash^{V[G_N]}_{\p_i'} x \in P(\dot T_i^N).
$$
Since $\p_i'$ is a dense suborder of $\p_i^N$ and $\dot T_i^N$ is 
a $\p_i^N$-name, we have that 
$$
q \Vdash^{V[G_N]}_{\p_i^N} x \in P(\dot T_i^N).
$$

\begin{claim}
$q \Vdash^{V[G_{i,N}]}_{\p_i^N} x \in P(\dot T_i^N)$.
\end{claim}

\begin{proof}
If not, then there is $r \le q$ in $\p_i^N$ such that 
$$
r \Vdash^{V[G_{i,N}]}_{\p_i^N} x \notin P(\dot T_i^N).
$$
Let $K$ be a $V[G_N]$-generic filter on $\p_i'$ with $r \in K$, and 
let $T := (\dot T_i^N)^K$. 
Since $r \le p \restriction i$, $p \restriction i \in K$. 
As $p \restriction i \in K$, in $V[G_N][K]$ we have that $p(i) = x \in P(T)$. 
On the other hand, $K \cap \p_i^N$ is a $V[G_{i,N}]$-generic filter on 
$\p_i^N$, and $T = (\dot T_i^N)^{K \cap \p_i^N}$. 
Since $r \in K \cap \p_i^N$, in $V[G_{i,N}][K \cap \p_i^N]$ 
we have that $x \notin P(T)$. 
But $P(T)$ is the same in $V[G_{i,N}][K \cap \p_i^N]$ and 
$V[G_N][K]$, which is a contradiction.
\end{proof}

Since $V[G_{i,N}]$ models that $q$ forces in $\p_i^N$ that $x \in P(\dot T_i^N)$, 
we can fix $s \in G_{i,N}$ such that 
$$
s \Vdash^V_{\p_i \cap N} q \Vdash^{V[\dot G_{i,N}]}_{\p_i^N} x \in 
P(\dot T_i^N).
$$
As $q \in \p_i^N$, fix $q_0 \in E_i^N$ such that 
$q = q_0 \restriction (i \setminus N)$. 
Then $q_0 \in \p_i / G_{i,N}$.  
By Lemma 1.4, $q_0$ and $s$ are compatible in $\p_i / G_{i,N}$, so fix 
$r_0 \le q_0, s$ in $\p_i / G_{i,N}$. 
By extending $r_0$ if necessary, we may assume that $r_0 \in E_i^N$. 
By Lemma 7.7, $r_0 \restriction N \in G_{i,N}$, and since $s \in N$ and $r_0 \le s$, 
$r_0 \restriction N \le s$.

\begin{claim}
$r_0 \Vdash^V_{\p_i} x \in P(\dot T_i)$.
\end{claim}

\begin{proof}
To prove this, let $H_i$ be a $V$-generic filter on $\p_i$ with $r_0 \in H_i$. 
We will show that $x \in P((\dot T_i)^{H_i})$. 
Let $H_{i,N} := H_i \cap N$, which is a $V$-generic filter on $\p_i \cap N$ 
by Lemma 7.5(2). 
Since $r_0 \le s$ and $s \in \p_i \cap N$, $s \in H_i \cap N = H_{i,N}$. 
By Lemma 1.1, 
$V[H_i] = V[H_{i,N}][H_i]$, and $H_i$ is a $V[H_{i,N}]$-generic 
filter on $\p_i / H_{i,N}$. 
Therefore $\pi_i[H_i \cap E_i^N]$ 
generates a $V[H_{i,N}]$-generic filter $H_i^N$ on 
$\p_i^N$, where $E_i^N$ and 
$\p_i^N$ are defined in $V[H_{i,N}]$ from $H_{i,N}$, 
instead of from $G_{i,N}$ as above. 
Moreover, $V[H_i] = V[H_{i,N}][H_i] = V[H_{i,N}][H_i^N]$. 
Since $r_0 \in H_i$, $r_0 \restriction (i \setminus N) \in H_i^N$. 
Now $r_0 \le q_0$, so $r_0 \restriction (i \setminus N) \le 
q_0 \restriction (i \setminus N) = q$. 
Hence $q \in H_i^N$. 
Let $T := (\dot T_i^N)^{H_i^N}$. 

Recall that 
$$
s \Vdash^V_{\p_i \cap N} q \Vdash^{V[\dot G_{i,N}]}_{\p_i^N} x \in 
P(\dot T_i^N).
$$
Since $s \in H_{i,N}$ and $q \in H_i^N$, it follows that $x \in P(T)$. 
By Notation 7.15, $T$ is equal to $\dot T_i^{L}$, where 
$L$ is the filter on $\p_i / H_{i,N}$ generated by $\pi_i^{-1}(H_i^N)$. 
By the definition of $H_i^N$, we have that 
$\pi_i[H_i \cap E_i^N] \subseteq H_i^N$, and therefore 
$H_i \cap E_i^N \subseteq \pi_i^{-1}(H_i^N) \subseteq L$. 
Since clearly $H_i \cap E_i^N$ generates the filter $H_i$, 
$H_i \subseteq L$. 
As $H_i$ and $L$ are both $V[H_{i,N}]$-generic filters 
on $\p_i / H_{i,N}$, it follows that $H_i = L$. 
So $T = (\dot T_i)^{H_i}$. 
It follows that in $V[H_i]$, $x$ is in $P((\dot T_i)^{H_i})$, which 
proves the claim.
\end{proof}

Since $r_0 \in \p_i$ and $r_0 \Vdash^V_{\p_i} x \in P(\dot T_i)$, by the 
definition of $\p_{i+1}$ it follows 
that $r_1 := r_0 \cup \{ (i,x) \}$ is in $\p_{i+1}$. 
Moreover, since $i \notin N$ and $r_0 \in E_i^N$, 
$$
r_1 \restriction N = r_0 \restriction N \in G_{i,N} \subseteq G_{i+1,N}.
$$
So $r_1 \in E_{i+1}^N$. 
Hence $r := r_1 \restriction ((i+1) \setminus N)$ is in $\p_{i+1}^N$. 
Since $r_0 \le q_0$, 
$$
r \restriction i = r_0 \restriction (i \setminus N) 
\le q_0 \restriction (i \setminus N) = q,
$$
and so 
$r \restriction i \le q \le p \restriction i$. 
Since $p(i) = x$, $r = (r \restriction i) \cup \{ (i,x) \} \le p$.
\end{proof}

\section{$\mathsf{IGMP}$ and the Continuum}

We now construct a model in which $\textsf{IGMP}$ holds and 
$2^\omega > \omega_2$. 
We begin with a ground model $V$ which satisfies $\textsf{GCH}$, 
in which $\kappa$ is a supercompact 
cardinal, and $\lambda \ge \kappa$ is a cardinal with cofinality at least 
$\omega_2$. 
We will define a forcing poset of the form 
$$
(\p \times \textrm{Add}(\omega,\lambda)) * 
\dot \q
$$
which forces that $\kappa = \omega_2$, $2^\omega = \lambda$, and 
$\textsf{IGMP}$.

Let $\p$ be the strongly proper collapse of $\kappa$ to become $\omega_2$ 
which was discussed in Section 5.  
Then $\p$ is strongly proper, $\kappa$-c.c., has size $\kappa$, 
and collapses $\kappa$ to become $\omega_2$. 
Since $\textsf{GCH}$ holds in the ground model, 
standard arguments show that $\p$ forces that 
$2^\omega$ and $2^{\omega_1}$ are bounded by $\kappa$ 
(in fact, they are both equal to $\kappa$).

Consider the product forcing $\p \times \textrm{Add}(\omega,\lambda)$. 
Note that the forcing poset 
$\textrm{Add}(\omega,\lambda)$ is the same in $V$ and $V^\p$. 
Since $\p$ forces that $\textrm{Add}(\omega,\lambda)$ is 
$\omega_1$-c.c., it follows that 
$\p \times \textrm{Add}(\omega,\lambda)$ is $\kappa$-c.c., and 
it has size $\lambda$. 
So by standard arguments, 
$\p \times \textrm{Add}(\omega,\lambda)$ 
forces that $2^\omega = 2^{\omega_1} = \lambda$.

Applying Corollary 6.8, 
let $\dot \q$ be a $\p \times \textrm{Add}(\omega,\lambda)$-name 
for a special iteration of length $\lambda$ 
which forces that every tree of height and size $\omega_1$ which has 
no branches of length $\omega_1$ is special. 
Then $\p \times \textrm{Add}(\omega,\lambda)$ forces that 
$\dot \q$ is $\omega_1$-c.c., and has size $\lambda$. 
So the forcing poset 
$$
(\p \times \textrm{Add}(\omega,\lambda)) * \dot \q
$$
is $\kappa$-c.c., and forces that 
$2^\omega = 2^{\omega_1} = \lambda$.

By Corollary 4.5, in order to prove that 
$(\p \times \textrm{Add}(\omega,\lambda)) * \dot \q$ forces $\textsf{IGMP}$, 
it suffices to show that 
$(\p \times \textrm{Add}(\omega,\lambda)) * \dot \q$ 
forces that for all 
sufficiently large regular cardinals $\theta \ge \omega_2$, 
there are stationarily many sets $N$ in $P_{\omega_2}(H(\theta))$ 
such that $N$ is internally unbounded and $\omega_1$-guessing.

\bigskip

Fix a regular cardinal $\theta \ge \kappa$ such that the forcing poset 
$(\p \times \textrm{Add}(\omega,\lambda)) * \dot \q$ 
is a member of $H(\theta)$. 
Let $\dot F$ be a $(\p \times \textrm{Add}(\omega,\lambda)) * \dot \q$-name 
for a function $\dot F : (H(\theta))^{<\omega} \to H(\theta)$. 
We will prove that $(\p \times \textrm{Add}(\omega,\lambda)) * \dot \q$ forces that there exists a set $N$ satisfying:
\begin{enumerate}
\item $N$ is in $P_{\kappa}(H(\theta))$;
\item $N \prec H(\theta)$;
\item $N$ is closed under $\dot F$;
\item $N$ is internally unbounded and $\omega_1$-guessing.
\end{enumerate}
This will complete the proof.\footnote{The proof actually shows the 
existence of stationarily many $N$ satisfying (1)--(4) which are 
internally stationary. 
The reason is that for all regular $\omega_2 \le 
\theta < \kappa$, $\p$ forces that $H(\theta)$ is internally stationary. 
In fact, we can strengthen this to internally club, 
by modifying the adequate set forcing $\p$ by simultaneously 
adding $\kappa$ many club subsets of $\omega_1$ to get that for 
all regular 
$\omega_2 \le \theta < \kappa$, $\p$ forces that $H(\theta)$ 
is internally club. 
This modification of adequate set forcing is beyond the scope of 
the paper. 
Alternatively, this stronger result could also be obtained by replacing $\p$ 
with the decorated sequence forcing of Neeman \cite{neeman}.}

Let us abbreviate $\textrm{Add}(\omega,\lambda)$ 
by $\textrm{Add}$.

\begin{notation}
Fix a $V$-generic filter $G \times H$ on $\p \times \textrm{Add}$, and 
fix a $V[G \times H]$-generic filter $I$ on $\q := (\dot \q)^{G \times H}$. 
Let $F := (\dot F)^{(G \times H) * I}$.
\end{notation}

We will prove that in $V[(G \times H) * I]$, there exists a set $N$ in 
$P_{\kappa}(H(\theta))$ such that 
$N \prec H(\theta)$, $N$ is closed under $F$, and $N$ is internally 
unbounded and $\omega_1$-guessing.

\begin{notation}
By the supercompactness of $\kappa$, fix in $V$ an elementary 
embedding $j : V \to M$ with critical point $\kappa$ such that 
$j(\kappa) > |H(\theta)|$ and $M^{|H(\theta)|} \subseteq M$.
\end{notation}
 
Note that by the closure of $M$, we have that 
$H(\theta)^V = H(\theta)^M$ and 
$j \restriction H(\theta)^V \in M$. 
Since $H(\theta)^V = H(\theta)^M$ and 
$\p \times \textrm{Add} \in H(\theta)^V$, 
it follows that 
$$
H(\theta)^{V[G \times H]} = H(\theta)^V[G \times H] = 
H(\theta)^M[G \times H] = H(\theta)^{M[G \times H]},
$$
and also 
$$
H(\theta)^{V[(G \times H)*I]} = H(\theta)^V[(G \times H)*I] = 
H(\theta)^M[(G \times H)*I] = H(\theta)^{M[(G \times H)*I]}.
$$

Since the critical point of $j$ is $\kappa$ and $\p \times \textrm{Add}$ 
is $\kappa$-c.c., it follows that 
$j \restriction (\p \times \textrm{Add})$ 
is a regular embedding of $\p \times \textrm{Add}$ 
into $j(\p \times \textrm{Add})$, as proven in Section 1. 
Hence $j[\p \times \textrm{Add}]$ is a regular suborder 
of $j(\p \times \textrm{Add})$, $j[G \times H]$ is a $V$-generic filter 
on $j[\p \times \textrm{Add}]$, and 
$V[G \times H] = V[j[G \times H]]$. 
So in $V[G \times H]$, we can form the forcing poset 
$j(\p \times \textrm{Add}) / j[G \times H]$.

\begin{notation}
Let $J$ be a $V[(G \times H)*I]$-generic filter on the forcing poset 
$j(\p \times \textrm{Add}) / j[G \times H]$.
\end{notation}

Then in particular, $J$ is a $V[G \times H]$-generic filter on 
$j(\p \times \textrm{Add}) / j[G \times H]$. 
So by Lemma 1.1, $J$ is a $V$-generic filter on $j(\p \times \textrm{Add})$, 
$J \cap j[\p \times \textrm{Add}] = j[G \times H]$, and 
$V[G \times H][J] = V[J]$. 
It follows that 
$J$ is an $M$-generic filter on $j(\p \times \textrm{Add})$, and 
$j[G \times H] \subseteq J$. 
So in $V[J]$, standard arguments show that 
we can extend $j$ to 
$$
j : V[G \times H] \to M[J]
$$
by letting 
$$
j(\dot a^{G \times H}) := j(\dot a)^{J}
$$
for any $\p \times \textrm{Add}$-name $\dot a$ in $V$.

As noted above, 
$j \restriction H(\theta)^V \in M$. 
Since $H(\theta)^{V[G \times H]} = H(\theta)^V[G \times H]$, 
the definition of $j$ above shows that 
$j \restriction H(\theta)^{V[G \times H]}$ is definable in $M[J]$ from 
$H(\theta)^V$, $j \restriction H(\theta)^V$, $G \times H$, and $J$. 
Hence $j \restriction H(\theta)^{V[G \times H]}$ is in $M[J]$.

\bigskip

Since $J$ is a $V[G \times H][I]$-generic filter on 
$j(\p \times \textrm{Add}) / j[G \times H]$, by the product lemma, 
$I$ is a $V[G \times H][J]$-generic filter on $\q$. 
But $V[G \times H][J] = V[J]$, so $I$ is a $V[J]$-generic filter on $\q$. 
Hence $I$ is an $M[J]$-generic filter on $\q$. 
Also by the product lemma, 
$$
V[J][I] = V[G \times H][J][I] = V[G \times H][I][J].
$$
Since $j \restriction \q$ is an isomorphism in $M[J]$ from $\q$ to $j[\q]$, 
$j[I]$ is an $M[J]$-generic filter on $j[\q]$.

As $J$ is a $V[G \times H]$-generic filter on 
$j(\p \times \textrm{Add}) / j[G \times H]$, 
it is also an $M[G \times H]$-generic filter on 
$j(\p \times \textrm{Add}) / j[G \times H]$. 
So by Lemma 1.1, $M[G \times H][J] = M[J]$. 
By the product lemma,
$$
M[J][I] = M[G \times H][J][I] = M[G \times H][I][J].
$$

\bigskip

Let 
$$
N_0 := j[H(\theta)^{V[G \times H]}]
$$
Note that since $j \restriction H(\theta)^{V[G \times H]}$ is in $M[J]$, 
$N_0$ is in $M[J]$. 
We would like to apply Theorem 7.1 in $M[J]$ 
to the special iteration $j(\q)$ and the model $N_0$. 
Now $N_0$ has the same cardinality as $H(\theta)^{V[G \times H]}$, which 
in turn has the same cardinality as $H(\theta)^V$. 
But $|H(\theta)|^V < j(\kappa)$, and in $M[J]$, $j(\kappa)$ is equal to $\omega_2$ 
and $\omega_1^M$ is preserved.  
It follows that in $M[J]$, $N_0$ has size $\omega_1$. 
Also clearly $\omega_1 \subseteq N_0$.

We claim that $N_0$ is an elementary substructure of 
$H(j(\theta))^{M[J]}$. 
Let $F^*$ be a Skolem function for the structure $H(\theta)^{V[G \times H]}$ 
in $V[G \times H]$. 
Since $H(\theta)^{V[G \times H]}$ is closed under $F^*$, it easily follows 
that $N_0$ is closed under $j(F^*)$. 
By the elementarity of $j$, $j(F^*)$ is a Skolem function for 
$H(j(\theta))^{M[J]}$ in $M[J]$, which proves the claim.

So all of the assumptions of Theorem 7.1 for $j(\q)$ and $N_0$ hold in $M[J]$. 
By Theorem 7.1, it follows that in $M[J]$, 
$N_0 \cap j(\q)$ is a regular suborder of $j(\q)$, and 
$N_0 \cap j(\q)$ forces that $j(\q) / \dot G_{N_0 \cap j(\q)}$ is 
$\omega_1$-c.c.\ and 
has the $\omega_1$-approximation property.

\begin{lemma}
In the model $M[J]$, 
$j[\q]$ is a regular suborder 
of $j(\q)$, and $j[\q]$ forces that $j(\q) / \dot G_{j[\q]}$ is $\omega_1$-c.c.\ 
and has the $\omega_1$-approximation property.
\end{lemma}

\begin{proof}
By the comments preceding this lemma, 
it suffices to show that $j[\q] = j(\q) \cap N_0$. 
Let $q \in \q$, and we will show that $j(q) \in j(\q) \cap N_0$. 
Since $\dot \q \in H(\theta)^V$, $\q \in H(\theta)^{V[G \times H]}$. 
So $q \in H(\theta)^{V[G \times H]}$. 
Hence $j(q) \in N_0$. 
Also $j(q) \in j(\q)$ by the elementarity of $j$. 
Conversely, let $q^* \in j(\q) \cap N_0$, and we will show that 
$q^* \in j[\q]$. 
Then by the definition of $N_0$, there is $q \in H(\theta)^{V[G \times H]}$ 
such that $j(q) = q^*$. 
Since $j(q) = q^* \in j(\q)$, it follows that $q \in \q$ by 
the elementarity of $j$. 
Hence $q^* \in j[\q]$.
\end{proof}

\bigskip

Since $j[I]$ is an $M[J]$-generic filter on $j[\q]$, we can form the forcing 
poset $j(\q) / j[I]$ in $M[J][I]$. 
In particular, this forcing poset is in $V[G \times H][I][J]$.

\begin{notation}
Let $K$ be a $V[G \times H][I][J]$-generic filter on $j(\q) / j[I]$.
\end{notation}
 
Since $V[G \times H][I][J] = V[J][I]$, $K$ is a 
$V[J][I]$-generic filter on $j(\q) / j[I]$. 
Hence $K$ is an $M[J][I]$-generic filter on $j(\q) / j[I]$. 
By Lemma 1.1, it follows that $K$ is an $M[J]$-generic filter on $j(\q)$, 
$K \cap j[\q] = j[I]$, and $M[J][I][K] = M[J][K]$. 
In particular, $j[I] \subseteq K$. 
By standard arguments, 
in $V[J][I][K]$ we can extend the elementary embedding 
$j : V[G \times H] \to M[J]$ to 
$$
j : V[G \times H][I] \to M[J][K]
$$
by letting 
$$
j(\dot a^I) := j(\dot a)^K
$$
for any $\dot a$ which is a $\q$-name in 
$V[G \times H]$. 
Note that since $j \restriction H(\theta)^{V[G \times H]}$ is in $M[J]$, 
and $H(\theta)^{V[G \times H][I]} = H(\theta)^{V[G \times H]}[I]$, 
it follows that 
$j \restriction H(\theta)^{V[G \times H][I]}$ is definable in $M[J][K]$ 
from the parameters $H(\theta)^{V[G \times H]}$, 
$j \restriction H(\theta)^{V[G \times H]}$, $I$, and $K$. 
Therefore $j \restriction H(\theta)^{V[G \times H][I]}$ is in $M[J][K]$.

\begin{lemma}
The pair 
$$
(M[G \times H][I],M[J][K])
$$
has the $\omega_1$-covering property and the 
$\omega_1$-approximation property.
\end{lemma}

\begin{proof}
We have that 
$$
M[J][K] = M[J][I][K] = M[G \times H][J][I][K] = 
M[G \times H][I][J][K].
$$
So it suffices to show that the pair 
$$
(M[G \times H][I],M[G \times H][I][J][K])
$$
has the $\omega_1$-covering property and the 
$\omega_1$-approximation property.

By Proposition 5.6, 
for any regular suborder $\p_0$ of $\p \times \textrm{Add}$, 
$\p_0$ forces that $(\p \times \textrm{Add}) / \dot G_{\p_0}$ is 
strongly proper on a stationary set. 
By the elementarity of $j$, the same is true of 
$j(\p \times \textrm{Add})$ in $M$. 
In particular, this is true in $M$ of the regular suborder $j[\p \times \textrm{Add}]$ 
of $j(\p \times \textrm{Add})$. 
It follows that in $M[G \times H]$, the forcing poset 
$j(\p \times \textrm{Add}) / j[G \times H]$ is strongly proper on a 
stationary set. 

In $M[G \times H]$, $\q$ is a special iteration, and hence 
is $\omega_1$-c.c., and therefore proper. 
By Theorem 5.5, since $I$ is an $M[G \times H]$-generic filter on $\q$, 
in $M[G \times H][I]$ the forcing poset 
$j(\p \times \textrm{Add}) / j[G \times H]$ is 
strongly proper on a stationary set. 
Therefore in $M[G \times H][I]$, $j(\p \times \textrm{Add}) / j[G \times H]$ 
has the 
$\omega_1$-covering property and the $\omega_1$-approximation property. 
Since $J$ is an $M[G \times H][I]$-generic filter on 
$j(\p \times \textrm{Add}) / j[G \times H]$, it follows that 
the pair 
$$
(M[G \times H][I],M[G \times H][I][J])
$$
has the $\omega_1$-covering property and the 
$\omega_1$-approximation property.

By Lemma 8.4, in the model $M[J] = M[G \times H][J]$, 
$j[\q]$ is a regular suborder 
of $j(\q)$, and $j[\q]$ forces that $j(\q) / \dot G_{j[\q]}$ is $\omega_1$-c.c.\ 
and has the $\omega_1$-approximation property. 
Also $j[I]$ is an $M[J]$-generic filter on $j[\q]$. 
So in the model $M[J][I]$, 
$j(\q) / j[I]$ is $\omega_1$-c.c.\ and has the 
$\omega_1$-approximation property. 
Since $K$ is an $M[J][I]$-generic filter on $j(\q) / j[I]$, 
the pair 
$$
(M[J][I],M[J][I][K])
$$
has the $\omega_1$-covering property and the 
$\omega_1$-approximation property. 
But $M[J] = M[G \times H][J]$, so the pair 
$$
(M[G \times H][J][I],M[G \times H][J][I][K])
$$
has the $\omega_1$-covering property and the 
$\omega_1$-approximation property. 
Since $M[G \times H][J][I] = M[G \times H][I][J]$, 
the pair 
$$
(M[G \times H][I][J],M[G \times H][I][J][K])
$$
has the 
$\omega_1$-covering property and the 
$\omega_1$-approximation property.

To summarize, we have shown that the pairs 
$$
(M[G \times H][I],M[G \times H][I][J])
$$
and 
$$
(M[G \times H][I][J],M[G \times H][I][J][K])
$$
both have the $\omega_1$-covering property and the 
$\omega_1$-approximation property. 
By Lemma 1.7, it follows that the pair 
$$
(M[G \times H][I],M[G \times H][I][J][K])
$$
has the $\omega_1$-covering property 
and the $\omega_1$-approximation property.
\end{proof}

Recall that we are trying to prove that in $V[G \times H][I]$, 
there is a set $N$ in $P_{\kappa}(H(\theta))$ such that 
$N \prec H(\theta)$, $N$ is closed under $F$, and 
$N$ is internally unbounded and $\omega_1$-guessing. 
In the model $V[J][I][K]$, we have an elementary embedding 
$j : V[G \times H][I] \to M[J][K]$. 
So by the elementarity of $j$, it suffices to prove that in $M[J][K]$, there 
exists a set $N$ satisfying:
\begin{enumerate}
\item[(a)] $N$ is in $P_{j(\kappa)}(H(j(\theta))^{M[J][K]})$;
\item[(b)] $N \prec H(j(\theta))^{M[J][K]}$;
\item[(c)] $N$ is closed under $j(F)$;
\item[(d)] $N$ is internally unbounded and $\omega_1$-guessing.
\end{enumerate}

\bigskip

Let 
$$
N := j[H(\theta)^{V[G \times H][I]}].
$$
We will show that $N$ is in $M[J][K]$, 
and $M[J][K]$ models that $N$ satisfies 
properties (a)--(d) above.

\bigskip

Since $j \restriction H(\theta)^{V[G \times H][I]} \in M[J][K]$, as 
observed above, $N \in M[J][K]$. 
Also by the elementarity of $j$, 
$N \subseteq j(H(\theta)^{V[G \times H][I]}) 
= H(j(\theta))^{M[J][K]}$. 
Now $N$ has the same cardinality as $H(\theta)^{V[G \times H][I]}$, which 
in turn has the same cardinality as $H(\theta)^V$. 
But $|H(\theta)|^V < j(\kappa)$, and in $M[J][K]$, 
$j(\kappa)$ is equal to $\omega_2$ 
and $\omega_1^M$ is preserved.  
It follows that in $M[J][K]$, $N$ has size $\omega_1$. 
Hence $N$ is in $P_{j(\kappa)}(H(j(\theta))^{M[J][K]})$.

We claim that $N$ is an elementary substructure of 
$H(j(\theta))^{M[J][K]}$. 
Let $F^*$ be a Skolem function for the structure $H(\theta)^{V[G \times H][I]}$ 
in $V[G \times H][I]$. 
Since $H(\theta)^{V[G \times H][I]}$ is closed under $F^*$, 
it easily follows that $N$ is closed under $j(F^*)$. 
By the elementarity of $j$, $j(F^*)$ is a Skolem function for 
$H(j(\theta))^{M[J][I]}$ in $M[J][I]$. 
So $N$ is an elementary substructure of $H(j(\theta))^{M[J][I]}$. 
The same argument shows that $N$ is closed under $j(F)$. 

\bigskip

It remains to show that $N$ is internally unbounded and $\omega_1$-guessing 
in $M[J][K]$. 
To show that $N$ is internally unbounded, let $a$ be a countable 
subset of $N$ in $M[J][K]$.  
Then $b := j^{-1}[a]$ is a countable subset of 
$H(\theta)^{V[G \times H][I]}$ in $M[J][K]$. 
Since the pair $(M[G \times H][I],M[J][K])$ has the $\omega_1$-covering 
property by Lemma 8.6, 
we can fix a countable set 
$c \subseteq H(\theta)^{V[G \times H][I]} = H(\theta)^{M[G \times H][I]}$ 
in $M[G \times H][I]$ such that $b \subseteq c$. 
But $\theta$ is regular and uncountable in $M[G \times H][I]$, so 
$c \in H(\theta)^{M[G \times H][I]}$. 
Hence $j(c) = j[c]$ is in $N$, $j(c)$ is countable, and 
$a = j[b] \subseteq j[c] = j(c)$.

To show that $N$ is $\omega_1$-guessing in $M[J][K]$, 
by Lemma 2.2 it suffices to show that the pair 
$$
(\overline{N},M[J][K])
$$
satisfies the $\omega_1$-approximation 
property, where $\overline{N}$ is the transitive collapse of $N$. 
Since $N$ is isomorphic to $H(\theta)^{V[G \times H][I]}$, which is transitive, 
we have that $\overline{N} = H(\theta)^{V[G \times H][I]} 
= H(\theta)^{M[G \times H][I]}$. 
Hence it suffices to show that the pair 
$$
(H(\theta)^{M[G \times H][I]},M[J][K])
$$
has the $\omega_1$-approximation property.

Let $d$ be a bounded subset of 
$H(\theta)^{M[G \times H][I]} \cap On = \theta$ 
in $M[J][K]$ which is countably approximated by 
$H(\theta)^{M[G \times H][I]}$. 
We will show that $d$ is in $H(\theta)^{M[G \times H][I]}$. 
We claim that $d$ is countably approximated by $M[G \times H][I]$. 
Consider a countable set $a$ in $M[G \times H][I]$. 
Then $a \cap \theta$ is a countable subset of $\theta$ in $M[G \times H][I]$, and 
hence is in $H(\theta)^{M[G \times H][I]}$. 
Since $d \subseteq \theta$, $a \cap d = (a \cap \theta) \cap d$. 
Since $d$ is countably approximated by $H(\theta)^{M[G \times H][I]}$, 
$a \cap d = (a \cap \theta) \cap d$ is in $H(\theta)^{M[G \times H][I]}$, and hence is in $M[G \times H][I]$. 
Thus $d$ is countably approximated by $M[G \times H][I]$. 
Since $(M[G \times H][I],M[J][K])$ has the $\omega_1$-approximation property 
by Lemma 8.6, 
$d \in M[G \times H][I]$. 
But $d$ is a bounded subset of $\theta$ and $\theta$ is regular 
in $M[G \times H][I]$, so $d \in H(\theta)^{M[G \times H][I]}$.

\bigskip

We conclude the paper with two questions. 

\begin{enumerate}
\item In Corollary 3.5, we proved that $\textsf{IGMP}$ implies $\textsf{SCH}$. 
Does $\textsf{GMP}$ imply $\textsf{SCH}$?
\item At the end of Section 2, we noted that $\textsf{GMP}$ is consistent with 
$\mathfrak p = \omega_1$. 
Does $\textsf{IGMP}$ imply $\mathfrak p > \omega_1$?
\end{enumerate}

\bibliographystyle{plain}
\bibliography{paper28}

\end{document}